\documentclass{article}

\usepackage{amssymb}
\usepackage{amsthm}
\usepackage{amscd}
\usepackage{amsmath}
\theoremstyle{plain}
\newtheorem{theorem}{Theorem}[section]
\newtheorem*{theorem*}{Theorem}

\newtheorem*{maintheorem*}{Main Theorem}
\newtheorem{prop}[theorem]{Proposition}
\newtheorem{corollary}[theorem]{Corollary}
\newtheorem{lemma}[theorem]{Lemma}

\newtheorem*{conjecture*}{Conjecture}

\theoremstyle{definition}
\newtheorem{definition}[theorem]{Definition}
\newtheorem*{definition*}{Definition}
\newtheorem{example}[theorem]{Example}
\newtheorem*{example*}{Example}
\newtheorem*{notation*}{Notation}
\newtheorem*{notation-conv*}{Notation and convention}
\newtheorem*{convention*}{Convention}
\newtheorem{remark}[theorem]{Remark}

\newcommand{\Z}{{\mathbb Z}}
\newcommand{\C}{{\mathbb C}}
\newcommand{\Q}{{\mathbb Q}}
\newcommand{\R}{{\mathbb R}}

\makeatletter
  
  \@addtoreset{equation}{section} 
\makeatother

\begin{document}
\title{Decompositions of Abelian surface and quadratic forms} 
\author{Shouhei Ma \\
{\small Graduate School of Mathematical Sciences, University of Tokyo   } \\
{\small E-mail: sma@ms.u-tokyo.ac.jp } }
\date{}
\maketitle

\begin{abstract}
When a complex Abelian surface can be decomposed into a product of two elliptic curves,  
how many decompositions does the Abelian surface admit ? 
We provide arithmetic formulae for the number of decompositions of a complex Abelian surface. 
\footnote[0]{Keywords : Abelian surface, elliptic curve, binary quadratic form, Atkin-Lehner involution.} 
\footnote[0]{MSC 2000 : 14K02, 14H52, 11E16.}
\end{abstract}

\section{Introduction and main results}

Throughout this paper, 
an {\it Abelian surface\/} means a complex Abelian surface.  

Let $A$ be an Abelian surface which can be decomposed into a product of two elliptic curves.  
In general,  
the choice of a decomposition of $A$ is not unique even up to isomorphism.   
In the present paper we study the number of decompositions of $A$. 
For this problem, 
there are pioneering works of Hayashida \cite{Ha} and Shioda-Mitani \cite{S-M} :  
Let $\rho (A)$ be the Picard number of $A$   
and let $T_{A}$ be the transcendental lattice of $A$,  
which is the orthogonal complement of the  N\'eron-Severi lattice in $H^{2}(A, {\Z})$. 
When $\rho (A)=4$ and $T_{A}$ is primitive, 
Shioda and Mitani, with  the cooperation of Hirzebruch,   
expressed the number of decompositions of $A$    
in terms of the class number of a certain imaginary quadratic order determined by $A$.  
On the other hand, Hayashida calculated the number of decompositions when $\rho (A)=3$,  
in connection with the number of principal polarizations.

It is natural to expect counting formulae for the decompositions for all decomposable Abelian surface, 
which complete the works of Hayashida and Shioda-Mitani.    
The purpose of this paper is to give such counting formulae uniformly by a lattice-theoretic method.  
Firstly 
we give precise definitions.

\begin{definition}\label{def: basic}
Let $A$ be an Abelian surface.

$(1)$ 
A {\it decomposition\/} of $A$ 
is an ordered pair $(E_{1}, E_{2})$ of elliptic curves in $A$   
such that   
the natural homomorphism $E_{1}\times E_{2} \to A$ is an isomorphism. 
The Abelian surface $A$ is {\it decomposable\/} 
if there exists a decomposition of $A$.

$(2)$ 
Two decompositions $(E_{1}, E_{2})$ and $(F_{1}, F_{2})$ of $A$ are {\it strictly isomorphic\/}  
if $E_{1}\simeq F_{1}$ and $E_{2}\simeq F_{2}$, 
or equivalently  
if there exists an automorphism $f$ of $A$ such that $f(E_{i}) = F_{i}$. 
Two decompositions 
$(E_{1}, E_{2})$ and  $(F_{1}, F_{2})$ 
of $A$ 
are {\it isomorphic\/}  
if $(E_{1}, E_{2})$ is strictly isomorphic to $(F_{1}, F_{2})$ or to $(F_{2}, F_{1})$. 
\end{definition}

There are known several criterions 
for an Abelian surface to be decomposable.  
For example, 
Abelian surfaces with Picard number $4$ are always decomposable \cite{S-M}.  
Let ${\rm Dec\/} (A)$ (resp. $\widetilde{{\rm Dec\/}} (A)$) 
be the set of isomorphism (resp. strictly isomorphism) classes of decompositions of $A$,    
and put   
\begin{equation}\label{decomp numbers}
\delta (A) := | \: {\rm Dec\/} (A) \: |  
\; \; \;  \; \; \text{and} \; \; \;  \; \; 
\widetilde{\delta }(A) := | \: \widetilde{{\rm Dec\/}} (A) \: | . 
\end{equation}
The number $\delta (A)$ is regarded as the number of decompositions of $A$,  
while $\widetilde{\delta }(A)$ is  considered as the number of decompositions counted with multiplicity. 
If we define    
\begin{equation}\label{self decomp numbers}
\delta _{0}(A) := \Bigl| \:  \{ E : {\rm elliptic\/} \: {\rm curve\/}, \: \: E \times E \simeq A \} / \simeq \; \Bigr| , 
\end{equation}
an obvious relation 
\begin{equation}\label{relation of two decomp numbers}
\widetilde{\delta }(A) = 2\delta (A) - \delta _{0}(A)
\end{equation}
holds. 
Hence the knowledge of $\widetilde{\delta }(A)$   
in addition to that of $\delta (A)$  
would enable us to study the decompositions of $A$ more closely.  

We shall express the numbers 
$\delta (A)$  and  $\widetilde{\delta }(A)$  
in terms of the arithmetic of the transcendental lattice $T_{A}$. 
Let $\mathcal{G}(T_{A})$ be the genus of $T_{A}$, 
i.e., the set of isometry classes of lattices isogenus to $T_{A}$. 
For an even lattice $T$,   
let $D_{T}$ be the discriminant form of $T$, 
which is a finite quadratic form associated to $T$.  
We have a natural homomorphism 
$O(T) \to O(D_{T})$ between the isometry groups.  
For a natural number $n>1$,    
let $\tau (n)$ be the number of the prime divisors of $n$.     
We  put  $\tau (1) :=1$. 
Our formula for $\delta (A)$ is stated as follows.

\begin{theorem}\label{main1}
Let $A$ be a decomposable Abelian surface. 
Then $2\leq \rho (A) \leq 4$  and  
the decomposition number $\delta (A)$ is given as follows.

$(1)$ 
When $\rho (A) = 2$,  one has $\delta (A) = 1$. 

$(2)$ 
When $\rho (A) = 3$,  one has $\delta (A) = 2^{\tau (N)-1}$ where $2N = -{\rm det\/}(T_{A})$. 

$(3)$ 
When $\rho (A) = 4$ and 
$ T_{A} \not\simeq \begin{pmatrix}  2n & 0  \\
                                                                 0 & 2n \end{pmatrix}, 
                                                                 \begin{pmatrix}  2n & n  \\
                                                                                               n  &  2n   \end{pmatrix}, n>1$, 
one has  
\begin{equation*}\label{main formula 1}
\delta (A) =     \sum_{T\in \mathcal{G}(T_{A})}  | \: O(D_{T})/O(T) \: |.                     
\end{equation*}

$(4)$ 
When $\rho (A) = 4$ and 
$ T_{A} \simeq \begin{pmatrix}  2n & 0  \\
                                                                 0 & 2n \end{pmatrix}  \text{or} 
                                                                 \begin{pmatrix}  2n & n  \\
                                                                                               n  &  2n   \end{pmatrix}, n>1$, 
one has  
\begin{equation*}\label{main formula 2}
\delta (A) = 
                   \left\{ \begin{array}{cl} 
                             (2^{-4}+2^{-\tau (n)-3}) \cdot | O(D_{T_{A}}) |    
								                        &    \;        \text{if} \: \: T_{A} \simeq \begin{pmatrix} 2n & 0  \\
                                                                                                                                                                                           0 & 2n  \end{pmatrix}  ,                    \\
                             3^{-2} \cdot (2^{-2}+2^{-\tau(n)}) \cdot | O(D_{T_{A}}) |
								                        &    \;        \text{if} \: \: T_{A} \simeq \begin{pmatrix} 2n & n  \\
                                                                                                                                                                                          n   & 2n  \end{pmatrix},  n: \text{odd},    \\
                               3^{-2} \cdot (2^{-2}+2^{-\tau(2^{-1}n)}) \cdot | O(D_{T_{A}}) |
								                        &    \;        \text{if} \: \: T_{A} \simeq \begin{pmatrix} 2n & n  \\
                                                                                                                                                                                          n   & 2n  \end{pmatrix},  n: \text{even}.    \\
                               \end{array} \right. 
\end{equation*}
\end{theorem}

On the other hand, 
the number $\widetilde{\delta }(A)$ is expressed in slightly different way  
(in the case of  $\rho (A) =4$). 
The set of  proper equivalence classes of {\it oriented\/} lattices isogenus to $T_{A}$
is denoted by $\widetilde{\mathcal{G}}(T_{A})$.

\begin{theorem}\label{main2}
Let $A$ be a decomposable Abelian surface.  
The strict decomposition number $\widetilde{\delta} (A)$ is given as follows.

$(1)$ 
If $\rho (A) = 2$,  then  $\widetilde{\delta }(A) = 2$. 
  
$(2)$ 
If $\rho (A) = 3$,  then   
\begin{equation*}\label{main formula for Picard number 3}
\widetilde{\delta }(A) = 
                                      \left\{ \begin{array}{cc} 
                                                 1,                         &    \;  N=1, \\ 
                                                 2^{\tau (N)},   &    \; N>1, 
                                                \end{array} \right. 
\end{equation*}
where $2N = -{\rm det\/}(T_{A})$. 

$(3)$ 
If $\rho (A) = 4$ and 
$ T_{A} \not\simeq \begin{pmatrix}  2n & 0  \\
                                                                 0 & 2n \end{pmatrix}, 
                                                                 \begin{pmatrix}  2n & n  \\
                                                                                               n  &  2n   \end{pmatrix}, n>1$, 
then   
\begin{equation*}\label{main formula for Picard number 4, wide}
\widetilde{\delta }(A) =  2^{-1} \cdot |  \widetilde{\mathcal{G}}(T_{A})  | \cdot | O(D_{T_{A}}) | .  
\end{equation*}     

$(4)$ 
If $\rho (A) = 4$ and 
$ T_{A} \simeq \begin{pmatrix}  2n & 0  \\
                                                                 0 & 2n \end{pmatrix}  \text{or} 
                                                                 \begin{pmatrix}  2n & n  \\
                                                                                               n  &  2n   \end{pmatrix}, n>1$, 
then $\widetilde{\delta }(A)  =  2\delta (A)$.  
\end{theorem}

The number $| O(D_{T_{A}}) |$ appearing in the case of $\rho (A)=4$ 
is calculated explicitly in Section $\ref{section 7}$. 
On the other hand, the numbers $| \widetilde{\mathcal{G}}(T_{A})  |$ and  $| \mathcal{G}(T_{A})  |$ 
are rather deep and classical quantities.  
The reader may consult to \cite{Co} 
for the calculations of these quantities.

When $\rho (A)=4$ and $T_{A}$ is primitive, 
we have two types of formulae for $\delta (A)$ (or for $\widetilde{\delta}(A)$) : 
Shioda and Mitani's ideal-theoretic formula (\cite{S-M}, see Theorem \ref{S-M}) and our lattice-theoretic formula. 
These two  formulae are unified by 
the classical relation between primitive binary forms and quadratic fields. 
In particular, 
the comparison of two formulae will lead to 
an expression of the number of genera in a class group 
in terms of the discriminant form (Corollary \ref{expression of class number}).


The counting formula in the case of $\rho (A)=3$ is known to Hayashida \cite{Ha}, 
who calculated $\delta (A)$ as the number of reducible principal polarizations.  
The number $N$ is defined in a different way in \cite{Ha}.  
Given a decomposition $(E_{1}, E_{2})$ of $A$ with $\rho (A)=3$,   
we will determine explicitly all other members of $\widetilde{{\rm Dec\/}} (A)$   
from $(E_{1}, E_{2})$.  
It turns out that the periods of the members of $\widetilde{{\rm Dec\/}} (A)$,  
defined as points of the modular curve $\Gamma _{0}(N) \backslash {\mathbb H}$,  
can be transformed to each other by the  action of the Atkin-Lehner involutions on $\Gamma _{0}(N) \backslash {\mathbb H}$ 
(Proposition \ref{geometric meaning of A-L involution}).

The rest of the paper is organized as follows. 
In Sect.$3$, 
we prove general  formulae for  $\delta (A)$ and $\widetilde{\delta }(A)$.   
Here 
Shioda's Torelli theorem for Abelian surfaces \cite{Sh} 
and the technique of discriminant form developed by Nikulin \cite{Ni} 
are applied.    
These weak formulae will be analyzed in more detail for each Picard number. 
The case of Picard number $2$ is well-known,   
and can also be derived immediately from the weak formula. 
The case of Picard number $3$ is treated in Sect.$4$,   
and  the case of Picard number $4$ is studied in Sect.$5$. 
In Sect.\ref{section 7}
we will calculate the order of  the isometry group $O(D_{L})$  for a rank $2$ even lattice $L$.  
This part is purely algebraic and may be read independently.

\section{Conventions}

Let $A$ be an Abelian surface. 
The N\'eron-Severi (resp. transcendental) lattice of $A$ 
is denoted by $NS_{A}$ (resp. $T_{A}$). 
The Picard number  of $A$ is denoted by $\rho (A)$.  
For a curve $C \subset A$ its class in $NS_{A}$ is written as $[ C]$. 
The {\it positive cone\/} $\mathcal{C}_{A}^{+}$ of $A$ is 
the connected component of 
the open set $\{ x \in NS_{A}\otimes {\R}, \; (x, x)>0 \} $ 
containing ample classes. 
The group of Hodge isometries of  
$T_{A}$ is denoted by 
$O_{Hodge}(T_{A})$.

Let $L$ be an  even lattice, 
i.e.,  a free ${\Z}$-module of finite rank 
equipped with a non-degenerate integral symmetric bilinear form $( , )$ 
satisfying $(l, l) \in 2{\Z}$ for all $l \in L$. 
The isometry group of $L$ 
is denoted by $O(L)$. 
Let $SO(L) := \{ \gamma \in O(L), \: {\rm det\/}(\gamma ) =1 \}$, 
which is of index at most $2$ in $O(L)$.  
For  an integer $n \in {\Z}$, 
$L(n)$ denotes the lattice $(L, n( , ))$. 
An even lattice $L$ is {\it primitive\/} 
if $L \not\simeq L'(n)$ for any even lattice $L'$ and $n>1$. 
On the other hand, 
a sublattice $M$ (resp. a vector $l$) of $L$ 
is said to be {\it primitive\/} if $L/M$ (resp. $L/{\Z}l$) is free.  
In the rest of the paper,  
the distinction between these two notions of primitivity 
will be clear from the context.

Let $L^{\vee} = {\rm Hom\/}(L, {\Z})$ be the dual lattice of $L$, 
which is canonically embedded in the quadratic space $L \otimes {\Q}$ and contains $L$. 
On the finite Abelian group $D_{L} = L^{\vee}/L$ 
a natural quadratic form $q_{L} : D_{L} \to {\Q}/2{\Z}$ is defined by 
$q_{L}(x+L, x+L) = (x, x)+2{\Z}$. 
This finite quadratic form 
$(D_{L}, q_{L})$,  often abbreviated as $D_{L}$,   
is called the {\it discriminant form\/} of $L$. 
A homomorphism 
$r_{L} : O(L) \to O(D_{L}, q_{L})$ is defined naturally.   
For a primitive sublattice $L$ of an even unimodular lattice $M$ 
with the orthogonal complement $L^{\perp}$, 
there exists a canonical isometry (cf. \cite{Ni})
\begin{equation}\label{disc form of orthogonal complement}
 (D_{L}, q_{L}) \stackrel{\simeq}{\longrightarrow} (D_{L^{\perp}}, -q_{L^{\perp}}). 
\end{equation}

Two even lattices $L$ and $M$ are {\it isogenus\/} 
if $L\otimes {\Z}_{p} \simeq M\otimes {\Z}_{p}$ 
for every prime number $p$ 
and ${\rm sign\/}(L) = {\rm sign\/}(M)$. 
By Nikulin's theorem \cite{Ni} 
$L$ and $M$ are isogenus if and only if 
$(D_{L}, q_{L}) \simeq (D_{M}, q_{M})$ and ${\rm sign\/}(L) = {\rm sign\/}(M)$. 
The {\it genus\/} $\mathcal{G}(L)$ of $L$ 
is the set of isometry classes of even lattices isogenus to $L$. 
On the other hand, 
the set of orientation-preserving isometry classes of 
{\it oriented\/} even lattices isogenus to $L$ 
is denoted by $\widetilde{\mathcal{G}}(L)$ 
and called the {\it proper genus\/} of $L$. 
Writing 
\begin{equation}\label{ambiguous class}
\mathcal{G}_{1}(L) := \{ M \in \mathcal{G}(L) \: | \: O(M) \not= SO(M) \} 
\end{equation} 
and 
$\mathcal{G}_{2}(L) := \mathcal{G}(L) - \mathcal{G}_{1}(L)$, 
we have 
\begin{equation}\label{genus and proper genus}
 |  \widetilde{\mathcal{G}}(L)  |  = 
|  \mathcal{G}_{1}(L)  |  + 2 |  \mathcal{G}_{2}(L)  |  . 
 \end{equation}

The {\it hyperbolic plane\/} 
is the rank $2$ even unimodular lattice 
\begin{equation}\label{hyperbolic plane}
U = {\Z}e + {\Z}f,   \; \; \;  
(e, e) = (f, f) = 0, \; (e, f) = 1.  
\end{equation} 
Throughout  the paper we fix this basis $\{ e, f \} $ for $U$.  
The orientation-reversing isometry 
\begin{equation}\label{mutation isometry of hyperbolic plane}
\iota _{0} : U \longrightarrow U, \: \: \: \: 
\iota _{0} (e) =f, \: \: \iota _{0} (f)=e 
\end{equation}
will be used several times.

\section{Weak formulae}

\subsection{A formula for $\delta (A)$} 
Let $A$ be an Abelian surface. 
For a decomposition $(E_{1}, E_{2})$ of $A$ 
we define an embedding 
$\varphi : U \hookrightarrow NS_{A}$ 
by 
\begin{equation}\label{embedding associated to decomposition}
\varphi (e) = [ E_{1} ], \: \: \: 
\varphi (f) = [ E_{2} ],  
\end{equation}
where $e,  f \in U$ are as defined in $(\ref{hyperbolic plane})$.  
Since   
$([ E_{i} ], [ E_{i} ])  = 0$ 
and 
$([ E_{1} ], [ E_{2} ]) = 1$, 
$\varphi $ is certainly an embedding of $U$.

\begin{definition}
Let 
\begin{equation}\label{the group}
\Gamma _{A} := 
r_{NS}^{-1} \Bigl( \lambda \circ r_{T}(  O_{Hodge}(T_{A}) ) \Bigr) \: 
\subset O(NS_{A}), 
\end{equation}
where 
$r_{NS} : O(NS_{A}) \to O(D_{NS_{A}})$ 
and 
$r_{T} : O(T_{A}) \to O(D_{T_{A}})$ 
are the natural homomorphisms, 
and 
$\lambda : O(D_{T_{A}}) \simeq O(D_{NS_{A}})$ 
is the isomorphism induced by the isometry 
$(D_{T_{A}}, q_{T_{A}}) \simeq (D_{NS_{A}}, -q_{NS_{A}})$
(see $(\ref{disc form of orthogonal complement})$).
\end{definition}

There is an obvious inclusion 
${\rm Ker\/}(r_{NS}) \cdot \{ \pm {\rm id\/} \} \subset \Gamma _{A}$.   
Let $O_{Hodge}(H^{2}(A, {\Z}))$ be the group of Hodge isometries of  $H^{2}(A, {\Z})$.  
By Nikulin's theorem (\cite{Ni} Corollary 1.5.2),  
the group $\Gamma _{A}$ can be written as 
\begin{equation}\label{image of Hodge isometry group}
\Gamma _{A} = 
{\rm Image\/} \: (\: O_{Hodge}(H^{2}(A, {\Z})) \longrightarrow O(NS_{A})\: ). 
\end{equation}

\begin{prop}\label{injective}
Let $(E_{1}, E_{2})$ and $(F_{1}, F_{2})$ be decompositions of an Abelian surface $A$ 
and let $\varphi $, $\psi $ be the corresponding embeddings of $U$. 
Then  $(E_{1}, E_{2})$ and $(F_{1}, F_{2})$ 
are isomorphic if and only if 
$\varphi \in \Gamma _{A} \cdot \psi $. 
\end{prop}

\begin{proof}
If $(E_{1}, E_{2})$ and $(F_{1}, F_{2})$ are strictly isomorphic, 
there exists an automorphism $f$ of  $A$  
satisfying $f(E_{i}) = F_{i}$. 
Then 
$f^{\ast}([F_{i}]) = [E_{i}]$ 
and 
$f^{\ast}|_{NS_{A}} \in \Gamma _{A}$ 
so that we have 
$\varphi \in \Gamma _{A} \cdot \psi $. 
Let $\iota _{0}$ be the isometry of $U$ defined in $(\ref{mutation isometry of hyperbolic plane})$. 
The embedding associated to the decomposition $(E_{2}, E_{1})$ is 
\begin{equation*} 
\varphi \circ \iota _{0} = 
\left( 
(\varphi \circ \iota _{0} \circ \varphi ^{-1}) |_{\varphi (U)} \oplus {\rm id\/}_{\varphi (U)^{\perp}}
\right) \circ \varphi 
\: \in \Gamma _{A} \cdot \varphi . 
\end{equation*}
Therefore 
$\varphi $ and $\psi $ are $\Gamma _{A}$-equivalent 
if $(E_{1}, E_{2})$ and $(F_{1}, F_{2})$ are isomorphic.

Conversely, 
suppose that 
$\varphi = \gamma \circ \psi $ 
for some isometry $\gamma \in \Gamma _{A}$. 
By $(\ref{image of Hodge isometry group})$ 
$\gamma $ can be extended to a Hodge isometry 
$\Phi : H^{2}(A, {\Z}) \to H^{2}(A, {\Z})$.   
When ${\rm det\/}(\Phi ) = 1$, 
Shioda's Torelli theorem (\cite{Sh} Theorem 1) 
assures the existence of an automorphism $f$ of $A$ 
such that $f^{\ast} = \Phi $ or $-\Phi $. 
Since $\Phi $ preserves the cone $\mathcal{C}_{A}^{+}$, 
we have $f^{\ast} = \Phi $. 
Then 
$f^{\ast}([F_{i}]) = [E_{i}]$ 
so that 
$f(E_{i}) = F_{i}$. 
On the other hand, when ${\rm det\/}(\Phi ) = -1$,  
consider the Hodge isometry 
\begin{equation*}\label{new isometry}
\Psi  : =  
\Bigl(       (\varphi \circ \iota _{0} \circ \varphi ^{-1})|_{\varphi (U)}  \oplus {\rm id\/}_{\varphi (U)^{\perp}}     \Bigr)   
\circ   \Phi :   H^{2}(A, {\Z}) \to H^{2}(A, {\Z}).  
\end{equation*}  
As above,  
there exists an automorphism $g$ of $A$ such that $g^{\ast} = \Psi $.  
Hence $(E_{1}, E_{2})$ is strictly isomorphic to $(F_{2}, F_{1})$.  
\end{proof}

Let 
\begin{equation*}\label{Emb}
{\rm Emb\/} (U, NS_{A}) 
\end{equation*} 
be the set of embeddings of $U$ into $NS_{A}$. 
By Proposition \ref{injective} 
an injective map 
\begin{equation}\label{main map} 
{\rm Dec\/} (A) \hookrightarrow  \Gamma _{A} \backslash {\rm Emb\/}(U, NS_{A}) 
\end{equation} 
is defined. 
To prove its surjectivity, 
we need the following well-known proposition. 
A non-zero vector $l \in NS_{A}$ is {\it isotropic\/} 
if $(l, l)=0$.

\begin{lemma}[cf.\cite{B-L}] \label{prim isotropic vector}
Every primitive isotropic vector of $NS_{A}$ 
contained in the closure of the cone $\mathcal{C}_{A}^{+}$ 
is the class of an elliptic curve in $A$. 
\end{lemma}

Now we have

\begin{prop}\label{bijective}
The map defined in $(\ref{main map} )$ is bijective. 
\end{prop}

\begin{proof}
It suffices to prove the surjectivity.  
Let us given an embedding $\varphi : U \hookrightarrow NS_{A}$. 
Composing with $-{\rm id\/}$ if necessary, 
we may assume that the vector $\varphi (e)$ is contained in the closure of $\mathcal{C}_{A}^{+}$.  
As $(\varphi (e), \varphi (f)) = 1$, 
the vector $\varphi (f)$ is also contained in the closure of $\mathcal{C}_{A}^{+}$.  
By the above Lemma \ref{prim isotropic vector}, 
there exist elliptic curves $E_{1}, E_{2}$ in $A$ 
such that $[E_{1}]=\varphi (e)$ and $[E_{2}]=\varphi (f)$.  
As  $([E_{1}], [E_{2}]) = 1$, 
we have $E_{1} \cap E_{2} = \{ 0 \} $ 
so that $(E_{1}, E_{2})$ is a decomposition of $A$. 
The embedding associated with $(E_{1}, E_{2})$ is $\varphi $. 
\end{proof}

\begin{corollary}\label{criterion of decomposability} 
An Abelian surface $A$ is decomposable 
if and only if $NS_{A}$ admits an embedding of the hyperbolic plane $U$. 
\end{corollary}

\begin{prop}\label{a trick}
Let $L$ be an even lattice satisfying $NS_{A} \simeq U \oplus L$. 
Then 
\begin{equation*} 
\Bigl| \: \Gamma _{A} \backslash {\rm Emb\/}(U, NS_{A}) \: \Bigr| = 
\sum_{M\in \mathcal{G}(L)} \Bigl| \: O_{Hodge}(T_{A}) \backslash O(D_{M})/O(M) \: \Bigr| .
\end{equation*}
\end{prop}

\begin{proof}
For each even lattice $M \in \mathcal{G}(L)$ 
there exists an embedding 
$\varphi _{M} : U \hookrightarrow NS_{A}$ 
with 
$\varphi _{M}(U)^{\perp} \simeq M$ 
by Nikulin-Kneser's uniqueness theorem (Corollary 1.13.3 of \cite{Ni}).   
We have the decomposition  
\begin{eqnarray*} 
                     \Gamma _{A} \backslash {\rm Emb\/}(U, NS_{A}) 
&       =      & 
                     \mathop{\bigsqcup}_{M \in \mathcal{G}(L)} 
					              \Gamma _{A} \backslash \{ \: \varphi : U \hookrightarrow NS_{A}, \; \varphi (U)^{\perp} \simeq M \:  \}  \\
&       =      & 
                     \mathop{\bigsqcup}_{M \in \mathcal{G}(L)} 
					              \Gamma _{A} \backslash  \left(  O(NS_{A}) \cdot \varphi _{M}  \right)  \\ 
& \simeq & 
                    \mathop{\bigsqcup}_{M \in \mathcal{G}(L)} 
                                  \Gamma _{A} \backslash O(NS_{A}) / O(M) . 
\end{eqnarray*}
We apply the homomorphism 
$r : O(NS_{A}) \to O(D_{NS_{A}})$,  
which is surjective  
by Nikulin's theorem (\cite{Ni} Theorem 1.14.2). 
Since ${\rm Ker\/}(r) \subset \Gamma _{A}$, 
we obtain  
\begin{eqnarray*} 
                      \Gamma _{A} \backslash O(NS_{A}) / O(M) 
& \simeq & 
                     r ( \Gamma _{A}) \backslash O(D_{NS_{A}}) /  r (O(M))     \\ 
& \simeq & 
                     r ( O_{Hodge}(T_{A}) ) \backslash O(D_{M}) /  r (O(M)). 
\end{eqnarray*}
\end{proof}

By Propositions \ref{bijective} and \ref{a trick} 
we obtain

\begin{prop}\label{weak formula1} 
Let $A$ be a decomposable Abelian surface 
and let $L$ be an even lattice satisfying $NS_{A} \simeq U \oplus L$. 
Then the decomposition number $\delta (A)$ is given by 
\begin{equation}\label{eqn: weak formula1}
\delta(A) = 
\sum_{M \in \mathcal{G}(L)} \Bigl| \: O_{Hodge}(T_{A}) \backslash O(D_{M})/O(M) \: \Bigr| .
\end{equation}
\end{prop}

This formula will be analyzed in more detail in the subsequent sections.

\begin{remark}
The decomposition number is related to 
the number of principal polarizations.  
For example see Hayashida \cite{Ha}.  
Herbert Lange taught the author that 
Peter Schuster also studied in his thesis 
the number of decompositions and principal polarizations by using class numbers of Hermitian forms.  
See \cite{La} and the references therein for more details.
\end{remark}


\subsection{A formula for $\widetilde{\delta} (A)$} 
Let $A$ be an Abelian surface with $\rho (A) = 4$. 
In this case, 
the transcendental lattice $T_{A}$ is a rank $2$ positive-definite even lattice 
and the Hodge structure of $T_{A}$ induces a natural orientation of $T_{A}$. 
An isometry of $T_{A}$ preserves the Hodge structure 
if and only if it preserves the orientation. 
Thus we have 
\begin{equation}\label{Hodge isometry and SO}
O_{Hodge}(T_{A}) = SO(T_{A}). 
\end{equation}
Since 
$(D_{T_{A}}, q_{T_{A}}) \simeq (D_{NS_{A}}, -q_{NS_{A}})$, 
the lattices $NS_{A}$ and $U \oplus T_{A}(-1)$ are isogenus.  
It follows from Nikulin-Kneser's uniqueness theorem that 
\begin{equation}\label{NS and T for Picard number 4}
NS_{A} \simeq U \oplus T_{A}(-1). 
\end{equation}
In particular, 
$A$ is always decomposable (\cite{S-M}).  
Let 
\begin{equation}\label{special group}
S\Gamma _{A} := \Gamma _{A} \cap SO(NS_{A}). 
\end{equation}
For a Hodge isometry $\Phi $ of $H^{2}(A, {\Z})$ 
we have ${\rm det\/} (\Phi ) = 1$ 
if and only if $\Phi |_{NS_{A}} \in S\Gamma _{A}$.  
With this fact in mind, 
we can prove the following proposition similarly as 
Propositions \ref{injective} and \ref{bijective}.

\begin{prop}\label{bijective widetilde}
 Suppose that $\rho (A)=4$. 
 For a decomposition $(E_{1}, E_{2})$ of $A$ 
 we define the embedding $\varphi : U \hookrightarrow NS_{A}$ by 
 the equation $(\ref{embedding associated to decomposition})$. 
 Then this assignment induces the bijection 
\begin{equation*}
\widetilde{{\rm Dec\/}} (A) \simeq  S\Gamma _{A} \backslash {\rm Emb\/}(U, NS_{A}).  
\end{equation*} 
\end{prop}

For each lattice $T \in \mathcal{G}(T_{A})$ 
we can find an embedding 
$\varphi _{T} : U \hookrightarrow NS_{A}$ 
with 
$\varphi _{T}(U)^{\perp} \simeq T(-1)$. 
Then, as like the proof of Proposition \ref{a trick}, 
\begin{equation*}
 S\Gamma _{A} \backslash {\rm Emb\/}(U, NS_{A}) 
 = 
 \mathop{\bigsqcup}_{T \in \mathcal{G}(T_{A})}  S\Gamma _{A} \backslash ( O(NS_{A})\cdot \varphi _{T}  ) .  
\end{equation*}
The orbit 
$O(NS_{A})\cdot \varphi _{T}$ 
is decomposed as 
\begin{equation}\label{decomp of orbit}
O(NS_{A})\cdot \varphi _{T} = 
SO(NS_{A})\cdot \varphi _{T} \: \cup \: SO(NS_{A})\cdot (\varphi _{T} \circ \iota _{0}), 
\end{equation} 
where $\iota _{0}$ is the isometry of $U$ defined by the equation $(\ref{mutation isometry of hyperbolic plane})$.

\begin{lemma}
We have 
$\varphi _{T} \circ \iota _{0} \in SO(NS_{A})\cdot \varphi _{T}$ 
if and only if 
$SO(T) \not= O(T)$.  
\end{lemma}

\begin{proof}
If 
$\varphi _{T} \circ \iota _{0} = \gamma \circ \varphi _{T}$  
for some 
$\gamma \in SO(NS_{A})$, 
this $\gamma $ can be written as 
\begin{equation*}
\gamma = ( \varphi _{T} \circ \iota _{0} \circ \varphi _{T}^{-1} ) |_{\varphi _{T}(U)} \oplus \gamma ' 
\end{equation*}
for some $\gamma ' \in O(T(-1)) = O(T)$. 
Then 
${\rm det\/} (\gamma ') = {\rm det\/}(\gamma ) \cdot {\rm det\/}(\iota _{0})^{-1} = -1$ 
so that 
$SO(T) \not= O(T)$.  
The converse is proved similarly. 
\end{proof}

Therefore we have 
\begin{equation*}
\left| \: S\Gamma _{A} \backslash ( O(NS_{A})\cdot \varphi _{T}  )  \: \right|   = 
                                      \left\{ \begin{array}{cl} 
                                                 \left| \: S\Gamma _{A} \backslash SO(NS_{A}) / SO(T)  \: \right| ,            
                                                        												 &     \: \:  {\rm if\/}   \;  T \in \mathcal{G}_{1}(T_{A}),   \\ 
                                    2 \cdot \left| \: S\Gamma _{A} \backslash SO(NS_{A}) / SO(T)  \: \right| ,            
                                                        												 &     \: \:  {\rm if\/}   \;  T \in \mathcal{G}_{2}(T_{A}),   
                                                \end{array} \right. 
\end{equation*}
where 
$\mathcal{G}_{i}(T_{A})$ 
are the subsets of $\mathcal{G}(T_{A})$ 
defined in $(\ref{ambiguous class})$. 
Now an imitation of the proof of Proposition \ref{a trick} 
yields the following formula involving the proper genus $\widetilde{\mathcal{G}}(T_{A})$.

\begin{prop}\label{weak formula 2}
Let $A$ be an Abelian surface with $\rho (A) = 4$. 
Then 
\begin{equation*}\label{eqn: weak formula2}
\widetilde{\delta}(A) = 
\sum_{T \in \widetilde{\mathcal{G}}(T_{A})} \Bigl| \: SO(T_{A}) \backslash O(D_{T_{A}}) / SO(T) \: \Bigr| .
\end{equation*}
\end{prop}
We will study this formula more closely 
in Section 5.

\section{The case of Picard number 3} 
\subsection{Counting formula}
Let $A$ be a decomposable Abelian surface with $\rho (A) = 3$. 
Then 
$NS_{A} \simeq U \oplus \langle -2N \rangle $ 
for some $N \in {\Z}_{>0}$. 
This natural number $N$ may be calculated by 
\begin{equation}\label{detNS for Picard number 3}
N = \frac{1}{2} {\rm det\/}(NS_{A}) = -\frac{1}{2} {\rm det\/}(T_{A}). 
\end{equation}
We also have the following.

\begin{prop}\label{geometric meaning of N}
Let $(E_{1}, E_{2})$ be a decomposition of $A$. 
Then 
\begin{equation}\label{eqn: geometric meaning of N}
N = 
{\rm min\/} \Bigl\{  \: {\rm deg\/}\phi \Bigr. \:  \: \Bigl|   \: \phi : E_{1} \to E_{2}  \:  \: {\rm isogeny\/}  \: \Bigr\} .  
\end{equation}
In particular, 
the right hand side of $(\ref{eqn: geometric meaning of N})$ 
is independent of the choice of $(E_{1}, E_{2})$.
\end{prop}

\begin{proof}
Let $l \in NS_{A}$ be a generator of the rank $1$ lattice  
\begin{equation*}
({\Z}[E_{1}] + {\Z}[E_{2}])^{\perp} \cap NS_{A}  \simeq \langle -2N \rangle .  
\end{equation*} 
The class 
$[E_{1}] + l + N[E_{2}] \in NS_{A}$ 
is a primitive isotropic vector contained in the closure of $\mathcal{C}_{A}^{+}$ 
so that 
there exists an elliptic curve $E$ in $A$ with 
$[E] = [E_{1}] + l + N[E_{2}]$.  
Since 
$([E], [E_{2}]) = 1$ (resp. $([E], [E_{1}]) = N$), 
the degree of the projection 
$E \to E_{1}$ (resp. $E \to E_{2}$) 
is $1$ (resp. $N$). 
Thus we obtain an isogeny 
$E_{1} \to E \to E_{2}$ 
of degree $N$. 

Conversely, 
let $\phi : E_{1} \to E_{2}$ be an arbitrary isogeny. 
Its graph $\Gamma \subset A$ is an elliptic curve satisfying 
$([\Gamma ], [E_{2}]) = 1$ and $([\Gamma ], [E_{1}]) = {\rm deg\/}\phi $. 
We can write 
$[\Gamma ] = [E_{1}] + al + ({\rm deg\/}\phi ) [E_{2}] $ 
for some $a \in {\Z}$. 
Then we have 
${\rm deg\/}\phi  = a^{2}N  \geq N$.  
\end{proof}

It follows that 
\begin{equation*}\label{relation of two numbers for Picard number 3}
\widetilde{\delta }(A) = 
                                      \left\{ \begin{array}{cc} 
                                                 \delta (A),                         &      \: \:        \text{if} \: \:     N=1,  \\ 
                                                 2\delta (A),                       &      \: \:        \text{if} \: \:     N>1.   
                                                \end{array} \right. 
\end{equation*}

\begin{prop}[cf. \cite{Ha}]\label{decomp number for Pic number 3}
Let $A$ be a decomposable Abelian surface with $\rho (A) = 3$ and ${\rm det\/}(T_{A}) = -2N$. 
Then $\delta (A) = 2^{\tau (N) -1}$. 
We also have 
\begin{equation*}\label{eqn: decomp number for Picard number 3}
\widetilde{\delta }(A) = 
                                      \left\{ \begin{array}{cc} 
                                                 1,                                              &     \: \:        \text{if} \: \:  N=1,  \\ 
                                                 2^{\tau (N) },                       &      \: \:       \text{if} \: \: N>1.   
                                                \end{array} \right. 
\end{equation*}
\end{prop}

\begin{proof}
The right hand side of the formula $(\ref{eqn: weak formula1})$ can be written as  
\begin{equation*}
\Bigl| \: O_{Hodge}(T_{A}) \backslash O(D_{\langle -2N \rangle}) / O(\langle -2N \rangle ) \: \Bigr| . 
\end{equation*}
As ${\rm rk\/}(T_{A}) = 3$ is odd,  
it follows from Appendix B of \cite{H-L-O-Y} that 
$O_{Hodge}(T_{A}) = \{ \pm {\rm id\/} \} $.  
The isometry group $O(\langle -2N \rangle )$ is clearly $\{ \pm {\rm id\/} \} $. 
Since 
$D_{\langle -2N \rangle} = \langle \frac{-1}{2N} \rangle \simeq {\Z}/2N{\Z}$,  
we have (cf. \cite{H-L-O-Y}) 
\begin{equation*}
\left| \: O(D_{\langle -2N \rangle}) \: \right|  =  
                                      \left\{ \begin{array}{cc} 
                                                 1,                                              &     N=1,  \\ 
                                                 2^{\tau (N) },                       &      N>1.   
                                                \end{array} \right. 
\end{equation*}
\end{proof}

Note that $\widetilde{\delta}(A)$ can be represented simply as 
\begin{equation*}
\widetilde{\delta}(A) = | O(D_{NS_{A}}) | = | O(D_{T_{A}}) | . 
\end{equation*}
Proposition \ref{decomp number for Pic number 3} was first proved by Hayashida \cite{Ha}.   
Hayashida defined the number $N$ as the minimal degree of isogeny $E \to F$,  
where $(E, F)$ is a decomposition of $A$.

\subsection{Construction of decompositions}

Let $A = E_{1}\times E_{2}$ be a decomposable Abelian surface with $\rho (A) = 3$. 
We shall construct representatives of ${\rm Dec\/} (A)$ from $(E_{1}, E_{2})$.  
Let $e := [E_{1}]$, $f := [E_{2}]$, 
and  $l$ be a generator of the lattice 
$\langle e, f \rangle ^{\perp} \cap NS_{A}$.  
Firstly 
we construct representatives of 
the quotient set $\Gamma _{A} \backslash {\rm Emb\/}(U, NS_{A})$ 
as follows. 
Let 
\begin{equation}\label{def of Sigma}
\Sigma := \Bigl\{  \; ( r_{\sigma}, s_{\sigma}) \; \Bigr. \Bigl|    \;     r_{\sigma}, s_{\sigma} \in {\Z}_{>0},  
                                                                                                                       \; ( r_{\sigma}, s_{\sigma}) =1, \;  
                                                                                                                                r_{\sigma}s_{\sigma}=N, \;
                                                                                                                               r_{\sigma} \leq s_{\sigma} \; \Bigr\} .
\end{equation}
We have $|\Sigma | = 2^{\tau (n)-1}$.  
For each $\sigma \in \Sigma $ 
choose integers $a_{\sigma}, b_{\sigma} \in {\Z}$ satisfying 
$a_{\sigma}r_{\sigma} + b_{\sigma}s_{\sigma} =1$ 
and put 
\begin{eqnarray}
e_{\sigma}  & := &   r_{\sigma}e                                +  s_{\sigma}f                               + l ,                                              \label{new emb 1}    \\
f_{\sigma}   & := &   b_{\sigma}^{2}s_{\sigma}e + a_{\sigma}^{2}r_{\sigma}f - a_{\sigma}b_{\sigma}l ,   \label{new emb 2}   \\
l_{\sigma}    & := &  2Nb_{\sigma}e                           -2Na_{\sigma}f                            + (b_{\sigma}s_{\sigma} - a_{\sigma}r_{\sigma})l.
\end{eqnarray}
The vectors  
$e_{\sigma}$, $f_{\sigma}$ 
define an embedding 
$\varphi _{\sigma} : U \hookrightarrow NS_{A} $ 
with 
$\varphi _{\sigma} ( U )^{\perp} = {\Z}l_{\sigma}$.

\begin{lemma}\label{reps of Emb set}
The set 
$\{ \varphi _{\sigma} \} _{\sigma \in \Sigma}$ 
of embeddings represents 
$\Gamma _{A} \backslash {\rm Emb\/}(U, NS_{A})$
completely. 
\end{lemma}

\begin{proof}
As 
$| \Sigma | 
 =   | \, \Gamma _{A} \backslash {\rm Emb\/}(U, NS_{A}) \, | 
 =   2^{\tau (N)-1}$, 
it suffices to show that  
$\varphi _{\sigma} \not\in \Gamma _{A} \cdot \varphi _{\sigma '}$ 
if 
$\sigma \not= \sigma '$. 
There are exactly two isometries of $NS_{A}$, 
say $\gamma  _{\sigma ', \sigma }^{+}$  and  $\gamma  _{\sigma ', \sigma }^{-}$,  
satisfying 
$\gamma  _{\sigma ', \sigma }^{\pm } \circ \varphi _{\sigma} = \varphi _{\sigma '}$ :  
\begin{equation*}
\gamma  _{\sigma ', \sigma }^{\pm } (e_{\sigma})  = e_{\sigma '}, \: \: \: 
\gamma  _{\sigma ', \sigma }^{\pm } (f_{\sigma})  = f_{\sigma '}, \: \: \: 
\gamma  _{\sigma ', \sigma }^{\pm } (l_{\sigma})  = \pm l_{\sigma '} . 
\end{equation*}
When $N$ is odd, 
$D_{NS_{A}} \simeq {\Z}/2N{\Z}$ is decomposed as 
\begin{equation}\label{decompositio of cyclic disc form}
D_{NS_{A}} \simeq {\Z}/2{\Z} \oplus \mathop{\bigoplus }_{i=1}^{\tau (N)} {\Z}/p_{i}^{e_{i}}{\Z}, 
\end{equation} 
where 
$N=\prod p_{i}^{e_{i}}$ is the prime decomposition of $N$. 
We write 
\begin{equation*}
\frac{l}{2N} = x_{0} + x_{1} + \cdots + x_{\tau (N)}   \; \in D_{NS_{A}} 
\end{equation*} 
with respect to this decomposition $(\ref{decompositio of cyclic disc form})$. 
Direct calculations show that 
\begin{equation*}
\frac{l_{\sigma}}{2N}  \equiv  (b_{\sigma}s_{\sigma}-a_{\sigma}r_{\sigma}) \frac{l}{2N} 
= x_{0} + \Bigl(  \sum_{p_{i}|r_{\sigma}} x_{i}\Bigr)  -  \Bigl(  \sum_{p_{j}|s_{\sigma}} x_{j}\Bigr) .
\end{equation*}
Now, if  
$r(\gamma  _{\sigma ', \sigma }^{\pm } ) \in \{ \pm {\rm id\/} \} \subset O(D_{NS_{A}})$, 
we must have 
$r_{\sigma} =  r_{\sigma '} $ or $\, r_{\sigma} =  s_{\sigma '} $,  
which implies that $\sigma = \sigma '$ by the definition of $\Sigma $.   
Thus we have $\sigma = \sigma '$ if $\varphi _{\sigma '} \in \Gamma _{A} \cdot \varphi _{\sigma}$.   
The argument if  $N$ is even is similar. 
\end{proof}

Next we find decompositions corresponding to the embeddings 
$\{ \varphi _{\sigma} \} $. 
Let $\phi : E_{1} \to E_{2}$ be an isogeny of degree $N$,   
the existence of which is guaranteed by Proposition \ref{geometric meaning of N}. 
Its kernel 
$G := {\rm Ker\/}(\phi) \subset E_{1}$ 
is a cyclic group of order $N$     
and is uniquely determined by the ordered pair $(E_{1}, E_{2})$.  
Let  
\begin{equation}\label{components of Kernel group}
G = \mathop{\bigoplus }_{i=1}^{\tau (N)}  G_{i}, \; \; \; \;  
|G_{i}| = p_{i}^{e_{i}}, 
\end{equation}
be the decomposition of G into $p$-groups and put 
\begin{equation}\label{decomp of Kernel group}
G_{\sigma , 1} := \mathop{\bigoplus }_{p_{i}|r_{\sigma}}  G_{i}, \: \: \:  \: \: 
G_{\sigma , 2} := \mathop{\bigoplus }_{p_{j}|s_{\sigma}}  G_{j}.   
\end{equation}
We have a canonical decomposition 
$G = G_{\sigma , 1} \oplus G_{\sigma , 2}$.  
If we denote 
\begin{equation}\label{def of new elliptic curve}
E_{\sigma , i} := E_{1} / G_{\sigma , i},  \; \; \; \;   i=1,2,   
\end{equation}
then the isogeny $\phi : E_{1} \to E_{2}$ can be factorized as  
\begin{equation}\label{factorization}
E_{1} \stackrel{\phi _{\sigma , i}^{+}}{\longrightarrow} 
E_{\sigma , i}  \stackrel{\phi _{\sigma , i}^{-}}{\longrightarrow} E_{2},  \; \; \; \;   i=1,2  
\end{equation}
with 
\begin{equation*}
{\rm deg\/}(\phi _{\sigma , 1}^{+}) = {\rm deg\/}(\phi _{\sigma , 2}^{-}) = r_{\sigma},      \: \: \: \:  
{\rm deg\/}(\phi _{\sigma , 2}^{+}) = {\rm deg\/}(\phi _{\sigma , 1}^{-}) = s_{\sigma}.   
\end{equation*}

Let  
$\widehat{\phi _{\sigma , i}^{\pm}}$  
be the dual isogeny of  $\phi _{\sigma , i}^{\pm}$.

\begin{lemma}\label{relation of isogenies}
We have 
$
\phi _{\sigma , 2}^{+} \circ \widehat{\phi _{\sigma , 1}^{+}}  =  
\widehat{\phi _{\sigma , 2}^{-}}  \circ  \phi _{\sigma , 1}^{-}  
:  E_{\sigma ,1} \to E_{\sigma , 2}$.   
\end{lemma}

\begin{proof}
Let 
\begin{equation*}
\varphi :=  
\phi _{\sigma , 2}^{+} \circ \widehat{\phi _{\sigma , 1}^{+}}  -  
 \widehat{\phi _{\sigma , 2}^{-}}  \circ  \phi _{\sigma , 1}^{-}  :   
 E_{\sigma ,1} \longrightarrow E_{\sigma , 2}. 
 \end{equation*}
 Since 
 \begin{equation*}
 \phi _{\sigma , 2}^{-} \circ  \varphi  =  
 \phi _{\sigma , 1}^{-} \circ \phi _{\sigma , 1}^{+} \circ \widehat{\phi _{\sigma , 1}^{+}}  -  
 r_{\sigma}\phi _{\sigma , 1}^{-}  =  
 \phi _{\sigma , 1}^{-} r_{\sigma}  -  r_{\sigma}\phi _{\sigma , 1}^{-}  
 = 0,  
 \end{equation*}
 then we have  
 $\varphi (E_{\sigma ,1}) \subset {\rm Ker\/}(\phi _{\sigma , 2}^{-})$.   
 By the discreteness of ${\rm Ker\/}(\phi _{\sigma , 2}^{-} )$ 
 we conclude that   
 $\varphi (E_{\sigma ,1}) = \{ 0 \} $. 
 \end{proof}

Compared with the factorizations $(\ref{factorization})$,  
Lemma \ref{relation of isogenies} represents a symmetry 
between the pairs 
$(E_{1}, E_{2})$ and  
 $(E_{\sigma , 1},  E_{\sigma , 2})$.

 \begin{prop}\label{they are decompositions}
 The homomorphism    
 \begin{equation}\label{hom giving decomposition}
\alpha _{\sigma} :=  \begin{pmatrix}  \widehat{\phi _{\sigma , 1}^{+}}   &   b_{\sigma}\widehat{\phi _{\sigma , 2}^{+}}    \\
                                                              -\phi _{\sigma , 1}^{-}                    &  a_{\sigma}\phi _{\sigma , 2}^{-}                      \end{pmatrix} 
								: E_{\sigma , 1} \times  E_{\sigma , 2}  
								\longrightarrow 
								  E_{1} \times  E_{2}  
\end{equation}
is an isomorphism.  
In other words,  the pair 
$(  \alpha _{\sigma} (E_{\sigma , 1}),   \alpha _{\sigma}(E_{\sigma , 2})  )$  
gives a decomposition of $A$. 
\end{prop}

\begin{proof}
We put 
\begin{equation*}
 \beta _{\sigma} :=  \begin{pmatrix}  a_{\sigma}\phi _{\sigma , 1}^{+}   &   -b_{\sigma}\widehat{\phi _{\sigma , 1}^{-}}    \\
                                                              \phi _{\sigma , 2}^{+}                     &  \widehat{\phi _{\sigma , 2}^{-}}                      \end{pmatrix} 
								:  E_{1} \times  E_{2}  
								  \longrightarrow 
								E_{\sigma , 1} \times  E_{\sigma , 2} . 								    
\end{equation*} 
With the aid of Lemma \ref{relation of isogenies} we can show that 
$\beta _{\sigma} \circ \alpha _{\sigma} = {\rm id\/}$ 
and 
$\alpha _{\sigma} \circ \beta _{\sigma} = {\rm id\/}$.   
\end{proof}

\begin{theorem}\label{rep of Dec}
Let $A=E_{1}\times E_{2}$ be a decomposable Abelian surface with $\rho (A) = 3$.   
Then the decompositions  
$\{ (  \alpha _{\sigma} (E_{\sigma , 1}),   \alpha _{\sigma}(E_{\sigma , 2})  )  \} _{\sigma \in \Sigma}$  
of $A$ defined by  $(\ref{def of new elliptic curve})$ and  $(\ref{hom giving decomposition})$  
represent ${\rm Dec\/}(A)$ completely.  
\end{theorem}

\begin{proof}
We  may  assume that $N >1$.   
Let  
$C_{\sigma , i}  := \alpha _{\sigma} (E_{\sigma , i}) \subset A$.   
By calculating the degrees of the projections 
$C_{\sigma , i} \to E_{j}$, 
we see that 
\begin{equation*} 
([C_{\sigma , 1}], [C_{\sigma , 2}]) = 
(e_{\sigma}, f_{\sigma}) 
\: \:  \: \text{or} \: \: \: 
(e_{\sigma}-2l, f_{\sigma}+2a_{\sigma}b_{\sigma}l),   
\end{equation*}
where 
$e_{\sigma},  f_{\sigma}$ 
are the vectors defined by  $(\ref{new emb 1})$ and $(\ref{new emb 2})$.    
If  
$([C_{\sigma , 1}], [C_{\sigma , 2}]) =  (e_{\sigma}, f_{\sigma})$,  
the embedding of $U$ associated to the decomposition    
$(C_{\sigma , 1},  C_{\sigma , 2})$   is  $\varphi _{\sigma}$.   
If  
$([C_{\sigma , 1}], [C_{\sigma , 2}]) =  (e_{\sigma}-2l, f_{\sigma}+2a_{\sigma}b_{\sigma}l)$,   
the corresponding embedding is   
\begin{equation*}
\Bigl(  {\rm id\/}_{\langle e, f \rangle }  \oplus  -{\rm id\/}_{\langle l \rangle }      \Bigr)  \circ  \varphi _{\sigma}  \: \: 
\in  \Gamma _{A} \cdot  \varphi _{\sigma} .   
\end{equation*}
Thus our claim follows from Lemma \ref{reps of Emb set}.  
\end{proof}

The construction of the decompositions $(E_{\sigma , 1}, E_{\sigma , 2})$ is related with the geometries of elliptic modular curves.   
Let 
\begin{equation*}\label{Hecke modular group}
\Gamma _{0}(N) := \left\{    \begin{pmatrix} a & b  \\
                                                                                        c & d  \end{pmatrix}   \in SL_{2}({\Z}),  \: \: c \equiv 0 \mod{N} \:   \right\} .   
\end{equation*}			
As is well-known,   
the congruence modular curve $\Gamma _{0}(N) \backslash \mathbb{H}$  
is the moduli space of elliptic curves with cyclic subgroups of order $N$.  																
To a decomposition $(E, F)$ of $A$  
we associate a point in  $\Gamma _{0}(N) \backslash \mathbb{H}$ 
by considering the the pair $(E, {\rm Ker\/}(\phi))$,  
where $\phi  : E \to F$ is an isogeny of the minimal degree $N$.   
Let us denote this point by 
\begin{equation*}\label{period of decomposition}
\omega (E, F) \; \in \Gamma _{0}(N) \backslash \mathbb{H}.  
\end{equation*}
When $N>1$,  
the Abelian group $G = ({\Z}/2{\Z})^{\tau (N)}$ 
acts on the curve $\Gamma _{0}(N) \backslash \mathbb{H}$ 
by the Atkin-Lehner involutions (\cite{L-N}, see also \cite{Kl}).  
A comparison of the definition of $(E_{\sigma , 1}, E_{\sigma , 2})$ 
and that of Atkin-Lehner involutions yields the following.  

\begin{prop}\label{geometric meaning of A-L involution} 
Let $A=E_{1}\times E_{2}$ be a decomposable Abelian surface with $\rho (A)=3$.   
Assume that $E_{1} \not\simeq E_{2}$.  
Then we have 
\begin{equation*}\label{Atkin-Lehner involution}
G\cdot \omega (E_{1}, E_{2}) 
= \{  \: \omega (E_{\sigma , 1}, E_{\sigma , 2}),  \; \omega (E_{\sigma , 2}, E_{\sigma , 1})  \:  \} _{\sigma \in \Sigma }.     
\end{equation*} 
\end{prop}
In other words,  
all members of $\widetilde{{\rm Dec \/}}(A)$ can be constructed from a given $(E_{1}, E_{2})$ 
by the action of the Atkin-Lehner involutions.

\begin{remark}
Before finishing this section, 
let us indicate a way of generalizing Proposition \ref{decomp number for Pic number 3} 
to a counting formula for the set 
\begin{equation*}\label{isom set of elliptic curves}
{\rm El\/}(A) := {\rm Aut\/}(A) \backslash \left\{  \:  E \subset A  \: \: \text{elliptic curve} \: \right\} .  
\end{equation*} 
If we denote  
\begin{equation*}\label{set of isotropic vectors}
I(NS_{A}) := \left\{  \:  {\Z}l \subset NS_{A},  \:     \text{primitive isotropic sublattice of rank}  1  \: \right\} ,  
\end{equation*}
then 
${\rm El\/}(A) $ is naturally identified with the quotient set 
${\rm Aut\/}(A) \backslash I(NS_{A})$.  
Let 
\begin{equation*}\label{symple auto group}
{\rm Aut\/}(A)_{0} = \{ \: f \in {\rm Aut\/}(A),\: \: f^{\ast}|_{T_{A}} = {\rm id\/}_{T_{A}} \: \}
\end{equation*}
be the group of symplectic automorphisms of $A$,  
which is of index $2$ in ${\rm Aut\/}(A)$. 
Firstly we study the set 
${\rm Aut\/}(A)_{0} \backslash I(NS_{A})$.  
The image of ${\rm Aut\/}(A)_{0}$ in $O(NS_{A})$ is equal to the group 
\begin{equation*}
SO(NS_{A})_{0}^{+} := \left\{  \:  \gamma \in SO(NS_{A}) \;   \right| \left.  \;  r_{NS}(\gamma) = {\rm id\/} \in O(D_{NS_{A}}), \;  
\gamma (\mathcal{C}_{A}^{+})  =   \mathcal{C}_{A}^{+}  \:   \right\} .  
\end{equation*}
By the theory of Baily-Borel compactification \cite{B-B},  
the set 
$SO(NS_{A})_{0}^{+} \backslash I(NS_{A})$ 
is canonically identified with 
the set of cusps of the modular variety associated to the orthogonal group $SO(NS_{A})_{0}^{+}$.   
Via the Clifford algebra construction for example, 
the group $SO(NS_{A})_{0}^{+}$ is shown to be isomorphic to 
the congruence modular group $\Gamma _{0}(N)$,     
and the isomorphism of groups induces that of the corresponding modular curves. 
Hence  
the set ${\rm Aut\/}(A)_{0} \backslash I(NS_{A})$ 
is identified with the set of $\Gamma _{0}(N)$-cusps, which is well-known.  
Now the number 
$|{\rm El\/}(A)|$ is calculated by looking 
the action of the group 
${\rm Aut\/}(A) / {\rm Aut\/}(A)_{0}  \simeq {\Z}/2{\Z}$ 
on the set of $\Gamma _{0}(N)$-cusps.     
This involution is calculated to be 
$r \mapsto -r$ for $r \in {\Q}$.  
\end{remark}

\section{The case of Picard number 4}  

\subsection{Counting formulae}

Let $A$ be an Abelian surface with $\rho (A) = 4$.  
As is well-known, 
the isomorphism class of $A$ is uniquely determined by 
the proper equivalence class of the transcendental lattice $T_{A}$ \cite{S-M}.  
Therefore it seems natural to express $\delta (A)$ and $\widetilde{\delta}(A)$ 
in terms of the arithmetic of $T_{A}$.  
Since $T_{A}$ is positive-definite of rank $2$, 
the group $SO(T_{A})$  is described completely as   
\begin{equation*}\label{description of SO}
SO(T_{A})  \simeq   \left\{ \begin{array}{cl} 
                                                 {\Z}/2{\Z},        &      \: \:  \: \: \:           
												                             \text{if} \:  \: T_{A} \not\simeq \begin{pmatrix}  2n &  0  \\
                                                                                                                                                                               0   &  2n \end{pmatrix} ,  \:  
                                           																		                                  \begin{pmatrix}  2n &   n  \\
                                                                                                                                                                                   n    &  2n \end{pmatrix} ,			\\ 
                                                 {\Z}/4{\Z},        &     \: \:  \: \: \:        
												                            \text{if} \: \: T_{A} \simeq \begin{pmatrix} 2n &  0  \\
                                                                                                                                                                     0   & 2n  \end{pmatrix} ,     \\ 
					   {\Z}/6{\Z},         &    \: \:  \: \: \:          
												                            \text{if}  \: \: T_{A} \simeq \begin{pmatrix} 2n &  n  \\
                                                                                                                                                                                                                    n   & 2n  \end{pmatrix} .   
                                                \end{array} \right. 
\end{equation*}

First  we consider the general case : 
$
T_{A} \not\simeq \begin{pmatrix}  2n &  0  \\
                                                                     0   &  2n \end{pmatrix} , 
									 \begin{pmatrix}  2n &   n  \\
                                                                     n    &  2n \end{pmatrix} 	
$.  
The group $SO(T_{A}) $  
consists of $\{  \pm {\rm id\/} \} $ 
and the genus $\mathcal{G}(T_{A})$ does not contain the lattices 
$\begin{pmatrix} 2n &  0  \\
                                    0   & 2n  \end{pmatrix}$,   
$\begin{pmatrix} 2n &  n  \\
                                   n   & 2n  \end{pmatrix}$  
because they are unique in their genera.

\begin{prop}\label{Pic number 4 generic}
Suppose that $\rho (A)=4$ and  
$T_{A} \not\simeq \begin{pmatrix}  2n &  0  \\
                                                                        0   &  2n \end{pmatrix} ,  \:  
                                        \begin{pmatrix}  2n &   n  \\
                                                                         n    &  2n \end{pmatrix} $.  
Then  
\begin{eqnarray*}
&  & 	\delta (A)                         =   \sum_{T\in \mathcal{G}(T_{A})} \left| \: O(D_{T})/O(T) \: \right| .                                      \\  																																											   
&  &   \widetilde{\delta }(A)   =   \frac{1}{2} \cdot | \: \widetilde{\mathcal{G}}(T_{A}) \: | \cdot | \: O(D_{T_{A}}) \: |.    \\                
&  &  \delta _{0}(A) =   \left\{ \begin{array}{cl} 
                                                 \frac{1}{2} \cdot | \: O(D_{T_{A}}) \: |,       \: \: 
           												 &      \text{if}    \:    \begin{pmatrix}  2 &  1  \\
                                                                                                                               1  &  2c \end{pmatrix}        \:  
															      \text{or}    \:    \begin{pmatrix}  2 &  0  \\
                                                                                                                                 0 & 2c \end{pmatrix} 	\in 	\mathcal{G}(T_{A})	,  \\  					  
                                                   0,          \: \:                  &     \text{otherwise} . 
                                                \end{array} \right. 
\end{eqnarray*}
\end{prop}

\begin{proof} 
The first two equalities are deduced immediately from Propositions \ref{weak formula1} and \ref{weak formula 2}.  
Note that ${\rm id\/} \not=  - {\rm id\/}$ in $O(D_{T_{A}})$ 
because $| D_{T_{A}} | >4$.  
For the third equality, we have 
\begin{eqnarray*}
\delta _{0}(A)  & = &  2\delta (A) - \widetilde{\delta }(A)                        \\  
                              & = &  \sum_{T\in \mathcal{G}_{1}(T_{A})} 
							                 \Bigl\{  \; 2 \cdot |  O(D_{T})/O(T)  | \: -  \: |  O(D_{T})/ \{ \pm {\rm id\/} \}  |  \; \Bigr\}    \\  
                              & = &   | \: \mathcal{G}_{0}(T_{A}) \: | \cdot  
                                              | \: O(D_{T_{A}}) / \{ \pm {\rm id\/} \} \: | , 
\end{eqnarray*} 
where 
$\mathcal{G}_{0}(T_{A}) := 
\{ \: T \in \mathcal{G}_{1}(T_{A}) \: | \: r_{T}(O(T)) = \{ \pm {\rm id\/} \}  \: \}$.  
Let 
$T \in \mathcal{G}_{0}(T_{A})$.  
Since 
$SO(T) \not= O(T)$,  
it follows from the classification of ambiguous binary forms (\cite{Ca}, Chapter 14.4) that 
$T$ is isometric to one of the following lattices : 
\begin{equation*}\label{ambiguous form}  
           \begin{pmatrix}  2a &  0  \\
                                             0  &  2c  \end{pmatrix}  ,       \:  
															     \:    \begin{pmatrix}  2a & a  \\
                                                                                                           a  & 2c  \end{pmatrix}  , \: \: \:  a, c \in {\Z}_{>0}. 
\end{equation*}
If $T = \begin{pmatrix}  2a &  0  \\
                                                       0  &  2c  \end{pmatrix}$,  
$T$ has the orientation-reversing isometry 
$\gamma = \begin{pmatrix}   1  &  0    \\
                                                          0  &  -1  \end{pmatrix}$.   
By the requirement that $r_{T}(\gamma ) \in \{ \pm {\rm id\/} \} $, 
either $a$ or $c$ must be  equal to $1$.  
When $T = \begin{pmatrix}  2a &  a  \\
                                                       a  &  2c  \end{pmatrix}$,  
$T$ admits the orientation-reversing isometry 
$\sigma = \begin{pmatrix}      -1  &  -1    \\
                                                           0   &   1  \end{pmatrix}$.   
Similarly $a$ or $4c-a$ must be equal to $1$.  
In both cases $T$ is isometric to  $\begin{pmatrix}  2 & 1  \\
                                                                                                    1  & 2c  \end{pmatrix}$.     
In conclusion,  
$\mathcal{G}_{0}(T_{A})$ is either empty or 
consists of only one class of the type 
           $\begin{pmatrix}  2  &  0  \\
                                               0  &  2c  \end{pmatrix} $  or  
															                                          $\begin{pmatrix}  2 & 1  \\
                                                                                                                                         1  & 2c  \end{pmatrix}$, 
$c>0$. 																																		 
\end{proof}

Next  we study the  case  
$T_{A} \simeq \begin{pmatrix}  2n &  0  \\
                                                                0   &  2n \end{pmatrix}$. 
The group $SO(T_{A}) $  
is the cyclic group ${\Z}/4{\Z}$ generated by the rotation 
$\gamma _{1}  = \begin{pmatrix}    0  &  -1    \\
                                                              1  &  0  \end{pmatrix}$,    
and 																	 
the group $O(T_{A}) $  
is the dihedral group of order $8$ generated by  
$\gamma _{1} $ and  the reflection 
$\gamma _{2}  = \begin{pmatrix}    1  &  0    \\
                                                                      0  &  -1  \end{pmatrix}$  
with the relations 
$\gamma _{1}^{4} = \gamma _{2}^{2} = (\gamma _{1}\gamma _{2})^{2} =1$.  
The genus 
$\mathcal{G}(T_{A}) = \widetilde{\mathcal{G}}(T_{A})$  
consists only of $T_{A}$.  
The discriminant form $D_{T_{A}}$  
is the group $({\Z}/2n{\Z})^{2}$ 
endowed with the quadratic form  $\begin{pmatrix}  (2n)^{-1}     &  0                  \\
                                                                                               0                &  (2n)^{-1}  \end{pmatrix}$.  
For brevity, 
the image of $\gamma _{1}$ in $O(D_{T_{A}})$ 
is again denoted by $\gamma _{1}$.                                                                            
Straight calculations yield

\begin{lemma}\label{normality 1} 
Let $T_{A}$ and $\gamma _{1}$ be as above and assume that $n>1$.    
Then the homomorphism $O(T_{A}) \to O(D_{T_{A}})$ is injective.  
For an isometry $\gamma \in O(D_{T_{A}})$ 
we have $\gamma ^{-1}\gamma _{1}\gamma = {\rm det\/}(\gamma) \gamma _{1}$.  
As a result, we have 
\begin{equation*}
| \:  \langle \gamma _{1} \rangle   \cdot \gamma \cdot  \langle \gamma _{1} \rangle  \: |  =   
                                  \left\{ \begin{array}{cl} 
							                      4 ,    \; \;     &   {\rm det\/}(\gamma)     =      \pm 1,			\\ 
                                                  8 ,    \; \;     &   {\rm det\/}(\gamma) \not= \pm 1.              																																										  
                                    \end{array} \right. 
\end{equation*} 
\end{lemma}

When $n>1$, 
considering the determinant ${\rm det\/}(\gamma ) \in {\Z}/2n{\Z}$  
for $\gamma \in O(D_{T_{A}})$  induces a surjective homomorphism 
\begin{equation}\label{determinant homomorphism}
O(D_{T_{A}}) \twoheadrightarrow ({\Z}/2{\Z})^{\tau (n)}  .   
\end{equation}
Then we have 

\begin{prop}\label{Pic number 4 exceptional 1}
Assume that $\rho (A)=4$ and  
$T_{A} \simeq \begin{pmatrix}  2n &  0  \\
                                                         0   &  2n \end{pmatrix}$.     
Then  
\begin{equation*}
\delta (A) = 
                   \left\{ \begin{array}{cl} 
							              1 ,                                                                                    &   \: \: \:   n=1,			\\ 
                                       (2^{-4}+2^{-\tau (n)-3}) \cdot | \: O(D_{T_{A}}) \: |      ,        &   \:  \: \:  n>1,            																																										  
                                    \end{array} \right. 
\end{equation*}
and  
\begin{equation*}
\widetilde{\delta} (A) = 
                   \left\{ \begin{array}{cl} 
							                 1,                                   &   \: \: \:   n=1,			\\ 
                                           2\delta (A),                &   \:  \: \:  n>1.            																																										  
                                    \end{array} \right. 
\end{equation*}
\end{prop}

\begin{proof}
We may assume that $n>1$.    
It follows from Proposition \ref{weak formula 2}  and Lemma \ref{normality 1} that  
\begin{equation*}
\widetilde{\delta} (A) = 
| \:  SO(T_{A}) \backslash O(D_{T_{A}}) / SO(T_{A})  \: |  =   
|O(D_{T_{A}})|\cdot \frac{2^{\tau(n)-1}+1}{2^{\tau(n)-1}} \cdot \frac{1}{8}                    
\end{equation*} 
The formula for $\delta (A)$ is derived similarly.  
\end{proof}

Finally 
we consider the  case  
$T_{A} \simeq \begin{pmatrix}  2n &  n  \\
                                                        n   &  2n \end{pmatrix}$.  
The group $SO(T_{A}) $  
is the cyclic group ${\Z}/6{\Z}$ generated by the rotation 
$\sigma _{1}  = \begin{pmatrix}    0  &  -1    \\
                                                                   1  &  1     \end{pmatrix}$,    
and 																	 
the group $O(T_{A}) $  
is the dihedral group of order $12$ generated by  
$\sigma _{1} $ and  the reflection  
$\sigma _{2}  = \begin{pmatrix}    -1  &  -1    \\
                                                                     0  &   1  \end{pmatrix}$  
with the relations 
$\sigma _{1}^{6} = \sigma _{2}^{2} = (\sigma _{1}\sigma _{2})^{2} =1$.  
The image of $\sigma _{1}$ in $O(D_{T_{A}})$ 
is again denoted by $\sigma _{1}$.    
The genus 
$\mathcal{G}(T_{A}) = \widetilde{\mathcal{G}}(T_{A})$  
consists only of $T_{A}$.

\begin{lemma}\label{matrix calculation}
If a matrix $\gamma \in GL_{2}({\Z}/n{\Z})$ 
satisfies ${\rm det\/}(\gamma)^{2} =  1$ and 
\begin{equation*}
{}^t\gamma 	\begin{pmatrix}        2    &     1                     \\
                                                               1    &     2       \end{pmatrix}   \gamma  
\equiv 																	   
                                \begin{pmatrix}        2    &     1                     \\
                                                                   1    &     2       \end{pmatrix}    \: \: \: 
{\rm mod\/} \: 
                      \begin{pmatrix}  2n{\Z}    &      n{\Z}              \\
                                                   n{\Z}       &   2n{\Z}        \end{pmatrix}  , 														
\end{equation*} 
then we have 
\begin{equation}\label{adjoint act}
\gamma ^{-1} \begin{pmatrix}   0   &    -1                     \\
                                                               1    &    1        \end{pmatrix}  \gamma  =  
{\rm det\/}(\gamma)  \cdot 
\begin{pmatrix}   \frac{{\rm det\/}(\gamma)-1}{2}    &                                          -1                     \\
                                                               1                                         &     \frac{{\rm det\/}(\gamma)+1}{2}        \end{pmatrix}. 
\end{equation} 
\end{lemma}

\begin{proof}
This lemma is proved by direct calculation. 
\end{proof}

The discriminant group $D_{T_{A}}$ contains the subgroup 
$n^{-1}T_{A}/T_{A}$ of index $3$.  
We consider the quadratic form on $n^{-1}T_{A}/T_{A}$ 
induced from the discriminant form.  
The basis of $T_{A}$ induces that of $n^{-1}T_{A}/T_{A}$,  
with respect to which the quadratic form is written as 
$\begin{pmatrix}        2n^{-1}    &    n^{-1}                     \\
                                      n^{-1}    &   2n^{-1}       \end{pmatrix} $.    
As the subgroup $n^{-1}T_{A}/T_{A}$ coincides with    
the subgroup $\{ x \in D_{T_{A}}  |  nx = 0 \} $,  
it is preserved by the action of $O(D_{T_{A}})$  
so that we have a natural homomorphism 
\begin{equation*}
\varphi : O(D_{T_{A}}) \to O(n^{-1}T_{A}/T_{A}) .  
\end{equation*}   
With respect to the basis of $n^{-1}T_{A}/T_{A}$,  
the isometry  $\varphi (\sigma _{1})$ is represented as 
$\begin{pmatrix}        0    &    -1                     \\
                                      1    &    1       \end{pmatrix} $.

\begin{lemma}\label{normality 2}
Let $T_{A}$,  $\sigma _{1}$, and $\varphi $ be as above and assume that $n>1$.      
Then the homomorphism $O(T_{A}) \to O(D_{T_{A}})$ is injective.  
For $\gamma \in O(D_{T_{A}})$  we have  
\begin{equation*}  
| \: \langle \sigma _{1} \rangle \cdot \gamma \cdot \langle \sigma _{1} \rangle \: |  =   
                                  \left\{ \begin{array}{cl}  
				               6 ,    \; \;     &      {\rm det\/}(\varphi (\gamma))     =      \pm 1,			\\  
                                                  18 ,    \; \;     &   {\rm det\/}(\varphi (\gamma)) \not=  \pm 1 .            																      \end{array} \right. 
\end{equation*} 
\end{lemma} 

\begin{proof}
The first assertion is proved immediately. 
We prove the second assertion.  
First consider the case $3|n$. 
The quadratic form on $n^{-1}T_{A}/T_{A} \subset D_{T_{A}}$ is degenerated.  
We can show that 
the natural homomorphism $\varphi $ 
is injective and that 
${\rm det\/}(\varphi(\gamma)) ^{2}=  1 \in {\Z}/n{\Z}$ 
for $\gamma \in O(D_{T_{A}})$.   
By applying Lemma \ref{matrix calculation} 
to $\varphi(\gamma)$ and $\varphi(\sigma _{1})$,    
we see that  
\begin{equation*}
| \: \langle \varphi(\sigma _{1}) \rangle \cdot  \varphi(\gamma )\cdot \langle \varphi(\sigma _{1}) \rangle \: |  =   
                                  \left\{ \begin{array}{cl} 
					     6 ,    \; \;     &      {\rm det\/}(\varphi (\gamma))     =      \pm 1,			\\ 
                                                  18 ,    \; \;     &   {\rm det\/}(\varphi (\gamma)) \not=  \pm 1 .            																																										  
                                    \end{array} \right. 
\end{equation*} 
Next consider the case $3 \nmid n$.  
By the orthogonal decomposition 
\begin{equation*}
D_{T_{A}} \: = \: (  nT_{A}^{\vee}/T_{A} ) \oplus  ( n^{-1}T_{A}/T_{A} )
                    \:  \simeq  \:  {\Z}/3{\Z} \oplus ({\Z}/n{\Z})^{2},  
\end{equation*} 
we have a canonical decomposition 
\begin{equation}\label{eqn: decomp of orthogonal group, exceptional}
O(D_{T_{A}})  \:  =  \:  O( nT_{A}^{\vee}/T_{A} ) \oplus  O( n^{-1}T_{A}/T_{A} ), 
\end{equation} 
so that $\varphi $ is the natural projection  with the kernel 
$O( nT_{A}^{\vee}/T_{A} ) \simeq {\Z}/2{\Z}$.   
When ${\rm det\/}(\varphi (\gamma))     =      \pm 1$,  
it follows from Lemma \ref{matrix calculation} that 
$\langle \sigma _{1} \rangle $ is a normal subgroup of $O(D_{T_{A}})$.  
When ${\rm det\/}(\varphi (\gamma))  \not=   \pm 1$,  
we have again by Lemma \ref{matrix calculation} that  
\begin{equation*}
| \langle \sigma _{1} \rangle \cdot \gamma \cdot \langle \sigma _{1} \rangle | 
=  | \langle \varphi (\sigma _{1}) \rangle \cdot \varphi (\gamma ) \cdot  \langle \varphi (\sigma _{1} ) \rangle | 
= 18.  
\end{equation*} 
\end{proof}

When $n>2$ is odd  (resp. even), 
considering the determinant ${\rm det\/}(\varphi (\gamma)) \in {\Z}/n{\Z}$ for $\gamma \in O(D_{T_{A}})$ 
induces a surjective homomorphism 
\begin{equation*}
O(D_{T_{A}}) \twoheadrightarrow ({\Z}/2{\Z})^{\tau (n)} \; \; \; 
(\text{resp}. \;  O(D_{T_{A}}) \twoheadrightarrow ({\Z}/2{\Z})^{\tau (2^{-1}n)}).    
\end{equation*}
Similarly as Proposition  \ref{Pic number 4 exceptional 1},  
we have

\begin{prop}\label{Pic number 4 exceptional 2}
Assume that $\rho (A)=4$ and  
$T_{A} \simeq \begin{pmatrix}  2n &  n  \\
                                                         n   &  2n \end{pmatrix}$.     
Then 
\begin{equation*}
\delta (A) = 
                   \left\{ \begin{array}{cl} 
							              1 ,                                                                                    &   \: \: \:   n=1,			   \\ 
                                        3^{-2} \cdot (2^{-2}+2^{-\tau(2^{-1}n)}) \cdot | \: O(D_{T_{A}}) \: |    ,           &   \:  \: \:  n : \text{even},    \\            		
                                        3^{-2} \cdot (2^{-2}+2^{-\tau(n)}) \cdot | \: O(D_{T_{A}}) \: |	,                           &   \: \: \:   n:\text{odd} >1                                               
                   \end{array} \right. 
\end{equation*}
and  
\begin{equation*}
\widetilde{\delta} (A) = 
                   \left\{ \begin{array}{cl} 
				      1,                                &   \: \: \:   n=1,			\\ 
                                           2\delta (A),                &   \:  \: \:  n>1.            																																										  
                                    \end{array} \right. 
\end{equation*}
\end{prop}

\subsection{Some conclusions}

From Propositions \ref{Pic number 4 generic}, \ref{Pic number 4 exceptional 1}, and \ref{Pic number 4 exceptional 2},  
we have

\begin{corollary}\label{isometric NS}
Let $A$ and $B$ be Abelian surfaces with Picard number $4$.  
If $NS_{A}$ is isometric to $NS_{B}$,  
or equivalently if $T_{A}$ is isogenus to $T_{B}$,  
then $\delta (A) = \delta (B)$ and $\widetilde{\delta}(A) = \widetilde{\delta}(B)$.  
\end{corollary}

\begin{corollary}\label{self-product}
Let $A$ be an Abelian surface with $\rho (A)=4$.  
Then we have $\delta _{0}(A) \not= 0$ 
if and only if 
$T_{A}$ is primitive and 
belongs to a principal genus, i.e.,  
$T_{A}$ is isogenus to either 
$\begin{pmatrix}  2 &  0  \\
                                    0   &  2c \end{pmatrix}$ 
or 
$\begin{pmatrix}  2 &  1  \\
                                    1   &  2c \end{pmatrix}$  
for some $c \in {\Z}_{>0}$.  
\end{corollary}

Note that Corollary \ref{self-product} can also be proved by Shioda-Mitani's ideal-theoretic method (\cite{S-M}, Section 4).  
It must be well-known to experts.   
For an odd prime number $p$ and $a \in ({\Z}/p{\Z})^{\times}$, 
let $\chi _{p}(a) = \Bigl( \frac{a}{p} \Bigr) $ be the Legendre symbol.  
For an odd number $n \equiv 1 \mod{4}$, 
we define 
\begin{equation*}
\chi _{2}(n) =  \left\{ \begin{array}{cl} 
				      1,                 &   \: \: \:   n \equiv 1 \mod{8},			\\ 
                                           -1,                &   \:  \: \:  n \equiv 5 \mod{8}.      			
                                \end{array} \right. 
\end{equation*}
By Proposition \ref{multiple-divisible} proved in Section $\ref{section 7}$ independently, 
we have the following.

\begin{corollary}\label{multiple of T_{A}}
Let $A$ be an Abelian surface with $\rho (A)=4$ and 
assume that 
$T_{A} \not\simeq 
\begin{pmatrix}  2n &  0  \\
                               0 &  2n \end{pmatrix} ,  
 \begin{pmatrix}  2n &  n  \\
                                 n &  2n \end{pmatrix}$.  
For a natural number $N>1$   
let $A_{N}$ be the Abelian surface with $T_{A_{N}}$ properly equivalent to the form $T_{A}(N)$.

$(1)$ The number $\widetilde{\delta}(A_{N})$ is divisible by $\widetilde{\delta}(A)$.  

$(2)$ If $N$ is coprime to ${\rm det\/}(T_{A})$,  
then 
\begin{equation}\label{decomp number divisible coprime case} 
\widetilde{\delta}(A_{N}) =  \widetilde{\delta}(A) \cdot 
2^{\tau (N)} \cdot N \cdot \mathop{\prod}_{p|N} \Bigl(  1- \frac{\chi _{p}(-{\rm det\/}(T_{A}))}{p} \Bigr) .  							  
\end{equation}

$(3)$ If $T_{A}$ is primitive and $N | {\rm det\/}(T_{A})^{a}$ for some $a \in {\Z}_{>0}$, 
then 
\begin{equation}\label{decomp number divisible totally non-coprime case}
\widetilde{\delta}(A_{N}) =  \widetilde{\delta}(A) \cdot 2^{\tau (N)} \cdot N.  							  
\end{equation}
\end{corollary}

We compare our formula with Shioda-Mitani's formula. 

\begin{theorem}[Shioda-Mitani  \cite{S-M}, Theorem 4.7]\label{S-M}
Suppose that $\rho (A)=4$ and that $T_{A}$ is primitive, 
i.e, $T_{A}$ is not isometric to $L(n)$ for any even lattice $L$ and $n > 1$.  
Let $\mathcal{O}$ be the unique order in an imaginary quadratic field 
with discriminant $d(\mathcal{O}) = -{\rm det\/}(T_{A})$, 
and $\mathcal{C}(\mathcal{O})$ be the ideal class group of $\mathcal{O}$.  
Then 
\begin{equation}\label{S-M formula}
\widetilde{\delta}(A) = | \mathcal{C}(\mathcal{O})  |. 
\end{equation}
\end{theorem}

\begin{corollary}\label{expression of class number}
Let $A$, $\mathcal{O}$ be as in Theorem \ref{S-M}.  
Then 
\begin{equation}\label{eqn: class number}
| \mathcal{C}(\mathcal{O})  | = 
\frac{1}{2} \cdot | \: \widetilde{\mathcal{G}}(T_{A}) \: | \cdot | \: O(D_{T_{A}}) \: | .
\end{equation}
In particular, 
the number of genera with discriminant $-{\rm det\/}(T_{A})$ 
is given by  
$\frac{1}{2} | O(D_{T_{A}}) |$. 
\end{corollary}

\begin{proof}
The first equality $(\ref{eqn: class number})$ follows from 
the comparison of the Shioda-Mitani formula $(\ref{S-M formula})$ 
and Propositions \ref{Pic number 4 generic}, \ref{Pic number 4 exceptional 1}, \ref{Pic number 4 exceptional 2}.  
The group of proper equivalence classes of 
primitive positive-definite rank $2$ even oriented lattices with determinant ${\rm det\/}(T_{A})$ 
is canonically isomorphic to the group $\mathcal{C}(\mathcal{O})$ (see  \cite{Co} Theorem 7.7). 
Hence the second assertion follows from the fact that 
all proper genera in a given class group 
consist of the same number of classes.  
\end{proof}

Of course, 
Corollary \ref{expression of class number} 
can be proved directly 
without going through decompositions of Abelian surfaces.  
Shioda-Mitani's formula is extended as follows.

\begin{corollary}\label{primitivity and conductor}
Let $A$ be an Abelian surface as in Theorem \ref{S-M}  
and suppose that ${\rm det\/}(T_{A}) \not= 3, 4$.  
For a natural number $N>1$  
let $A_{N}$ be the Abelian surface as in Corollary \ref{multiple of T_{A}} 
and $\mathcal{O}_{N}$ be the order with discriminant $d(\mathcal{O}_{N}) = -{\rm det\/}(T_{A}(N))$.  
Then 
\begin{equation*}
\widetilde{\delta}(A_{N}) = 2^{\tau (N)} \cdot | \: \mathcal{C}(\mathcal{O}_{N}) \: |.  
\end{equation*}
\end{corollary} 

\begin{proof}
By the assumption, 
we have 
$\widetilde{\delta}(A) = |  \mathcal{C}(\mathcal{O})  |$ 
for the order  $\mathcal{O}$ with $d(\mathcal{O}) = -{\rm det\/}(T_{A})$.   
Then our assertion follows from the comparison of 
the equations $(\ref{decomp number divisible coprime case})$, $(\ref{decomp number divisible totally non-coprime case})$  
and \cite{Co} Corollary 7.28.  
\end{proof}

An ideal-theoretic proof of this corollary is also available.

\subsection{Abelian surfaces with decomposition number $1$}

We shall study  
Abelian surfaces with $\rho (A)=4$ and $\delta (A)=1$.  
Such Abelian surfaces can be classified as follows : 
\begin{eqnarray*}
& ({\rm i\/}) &     T_{A} \: \text{is primitive}, \;  \widetilde{\delta}(A)=1.  \\ 
& ({\rm ii\/}) &    T_{A} \: \text{is primitive},  \; \widetilde{\delta}(A)=2.  \\ 
& ({\rm iii\/}) &   T_{A} \: \text{is not primitive}, \;  \widetilde{\delta}(A)=2.   
\end{eqnarray*} 
In the below 
we see that there are exactly thirteen, twenty-nine, four Abelian surfaces 
in the classes $({\rm i\/})$, $({\rm ii\/})$, $({\rm iii\/})$ respectively.

\begin{example}\label{only one symmetric decomposition}
Let $A$ be an Abelian surface such that $\rho (A)=4$ and   
${\rm Dec\/} (A) = \{ (E, E) \} $ for an elliptic curve $E$.    
Then $E$ is isomorphic to one of the following thirteen elliptic curves : 
\begin{eqnarray*}\label{list of elliptic curves}
&                 & \Bigl\{   \: \:  E(\frac{1+ \sqrt{\Delta}}{2}) \: \:  \Bigr. \Bigl| \; \; 
                                                    \Delta = -3, -7, -11, -19, -27, -43, -67, -163 \: \Bigr\}    \\
& \sqcup &  \Bigl\{   \: E( \sqrt{\Delta} ) \: \Bigr. \Bigl| \: \: 
                                                    \Delta = -1, -2, -3, -4, -7 \: \Bigr\} , 
\end{eqnarray*} 
where $E(\tau ) = {\C}/ {\Z}+ {\Z}\tau $.  
\end{example} 

\begin{proof}
By Corollary \ref{self-product} and Shioda-Mitani formula $(\ref{S-M formula})$,  
we have $\widetilde{\delta}(A) = 1$ 
if and only if 
$T_{A}$ is primitive and the class number $| \mathcal{C}(\mathcal{O})  |$ 
of the corresponding order $\mathcal{O}$ is equal to $1$.  
As a result of Heegner-Baker-Stark's theorem,    
$T_{A}$ is one of the following lattices (see \cite{Co} Theorem 7.30 (ii)) :  
\begin{eqnarray}\label{binary forms with class number 1}
&  &   
\begin{pmatrix}  2 &  1  \\
                                 1 &  2 \end{pmatrix} , 
\begin{pmatrix}  2 &  0  \\
                                 0 &  2 \end{pmatrix} , 
\begin{pmatrix}  2 &  1  \\
                                 1 &  4 \end{pmatrix} , 
\begin{pmatrix}  2 &  0  \\
                                 0 &  4 \end{pmatrix} , 
\begin{pmatrix}  2 &  1  \\
                                 1 &  6 \end{pmatrix} , 
\begin{pmatrix}  2 &  0  \\
                                 0 &  6 \end{pmatrix} , 								 
\begin{pmatrix}  2 &  0  \\
                                 0 &  8 \end{pmatrix} , 
\nonumber\\   
&  & 
\begin{pmatrix}  2 &  1  \\
                                 1 &  10 \end{pmatrix} , 
\begin{pmatrix}  2 &  1  \\
                                 1 &  14 \end{pmatrix} , 
\begin{pmatrix}  2 &  0  \\
                                 0 &  14 \end{pmatrix} , 
\begin{pmatrix}  2 &  1  \\
                                 1 &  22 \end{pmatrix} , 
\begin{pmatrix}  2 &  1  \\
                                 1 &  34 \end{pmatrix} , 
\begin{pmatrix}  2 &  1  \\
                                 1 &  82 \end{pmatrix} .  
\end{eqnarray}
By Shioda-Mitani theory 
we can determine the elliptic curve $E$ explicitly from the transcendental lattice $T_{A}$ 
: see \cite{S-M},  Section 3.  
\end{proof}

On the other hand, 
for those $A$ with primitive $T_{A}$ and $\delta (A)=1$, $\widetilde{\delta}(A)=2$  
we have the following.

\begin{example}\label{decomp number 1 non-self-product}
There exist natural one-to-one correspondences 
between the following two sets : 

$({\rm a\/})$ The set of isomorphism classes of Abelian surfaces $A$ with $\rho (A)=4$ 
such that $T_{A}$ is primitive and ${\rm Dec\/}(A) = \{  (E_{1},  E_{2})  \}$, $E_{1} \not\simeq E_{2}$.  


$({\rm b\/})$ The set of imaginary quadratic order with class number $2$.  
\end{example}

\begin{proof}
For an Abelian surface $A$ in the set $({\rm a\/})$,  
we associate the order $\mathcal{O}$ with discriminant $d(\mathcal{O}) = -{\rm det\/}(T_{A})$.  
By Shioda-Mitani formula we have $| \mathcal{C}(\mathcal{O}) | =2$. 
Then $T_{A}$ corresponds to the non-trivial element of the class group $\mathcal{C}(\mathcal{O})$.   
\end{proof}


Imaginary quadratic fields with class number $2$ 
are also classified by \cite{M-W}, \cite{St}.  
As a result, 
imaginary quadratic orders with class number $2$ are given by the following twenty-nine discriminants :  
\begin{eqnarray*}
-d(\mathcal{O})  & = &  15,  20,  24,  32,  35,  36,  40,  48,  51,  52,  60,  64,  75,  88,  91,  99,  \\ 
                              &    &   100, 112, 115,  123,  147,  148,  187,  232,  235,  267,  403,  427,  748. 
\end{eqnarray*}
Enumeration of the non-principal forms is a rather straightforward task 
and is left to the reader.

\begin{example}\label{dec number 1 non primitive}
Let $A$ be an Abelian surface with $\rho (A)=4$ and 
assume that $T_{A}$ is {\it not\/} primitive. 
If $\delta (A)=1$,  
then $A$ is isomorphic to one of the following four Abelian surfaces :  
\begin{eqnarray*}
E(\sqrt{-1})\times E(2\sqrt{-1}),    &  &
E(\tau _{1})\times E(2\tau _{1}),     \\
E(\tau _{1})\times E(3\tau _{1}),   &  & 
E(\tau _{2})\times E(2\tau _{2}),    
\end{eqnarray*} 
where 
$\tau _{1}= \frac{1+\sqrt{-3}}{2}$,  
$\tau _{2}= \frac{1+\sqrt{-7}}{2}$,  
and  
$E(\tau)={\C}/{\Z}+{\Z}\tau $.   
\end{example}

\begin{proof}
As $T_{A}$ is not primitive, 
we see from Corollary \ref{self-product} that  $\widetilde{\delta} (A)=2$.  
First consider the case 
$T_{A} \simeq  
\begin{pmatrix}  2n &  0  \\
                                0 &  2n \end{pmatrix} $ with  $n>1$.  
Since the actions of 
$SO(T_{A})$ preserve the fibres of the determinant homomorphism $(\ref{determinant homomorphism})$,                                 
we have $\tau (n)=1$ 
so that  $SO(T_{A})$ is a normal subgroup of $O(D_{T_{A}})$ by Lemma \ref{normality 1}.  
Thus we have $| O(D_{T_{A}}) | = 8$,  
which implies by Theorem \ref{order of isom grp of disc form}                           
that $n=2$.  
In the case of 
$T_{A} \simeq  
\begin{pmatrix}  2n &  n  \\
                                n &  2n \end{pmatrix} $ with  $n>1$,  
we have $n=2, 3$ by similar argument.

Next consider the case                                
$T_{A} \not\simeq  
\begin{pmatrix}  2n &  0  \\
                                0 &  2n \end{pmatrix} ,  
\begin{pmatrix}  2n &  n  \\
                                n &  2n \end{pmatrix}$.  
Let $B$ be the Abelian surface such that 
$T_{B}$ is primitive and $T_{B}(n)$ is properly equivalent to $T_{A}$,  $n>1$.  
By Proposition \ref{multiple of T_{A}} 
the number $\widetilde{\delta} (B)$ divides $\widetilde{\delta} (A) = 2$ 
so that $\widetilde{\delta} (B) = 1$ or $2$.  
However, 
the case that $\widetilde{\delta} (B) = 2$   
is impossible by Corollary \ref{multiple of T_{A}} $(2)$, $(3)$.  
Thus we have $\widetilde{\delta} (B) = 1$ 
so that 
$T_{B}$ is one of the lattices in the list $(\ref{binary forms with class number 1})$.   
It follows from Corollary \ref{primitivity and conductor} that $n=2$ and 
$T_{B} =  
\begin{pmatrix}  2 & 1  \\
                              1 & 4 \end{pmatrix} $ 
is the only case.  
\end{proof}

\section{Order of isometry group of discriminant form}\label{section 7} 
The aim of this section is to calculate the order of the group 
$O(D_{L})$ for a rank $2$ even lattice $L$.  
This section may be read independently of the previous sections. 
By the orthogonal decomposition 
$D_{L} = \mathop{\oplus}_{p} D_{p}$ 
where $D_{p}$ is the $p$-component for the prime number $p$,  
we have the canonical decomposition 
\begin{equation}\label{sect.7 decomp of isometry group}
O(D_{L}) = \mathop{\bigoplus}_{p} O(D_{p}). 
\end{equation}
Thus we first calculate $|O(D_{p})|$ for each prime number $p$, 
and then put them together to $|O(D_{L})|$.  
Throughout this section,   
a {\it finite quadratic form\/} means a finite Abelian group $D$ 
endowed with a quadratic form 
$q : D \to {\Q}/2{\Z}$ 
such that the associated bilinear form 
$b : D \times D \to {\Q}/{\Z}$ 
is non-degenerate. 
For a finite quadratic form $(D, q)$ on a $p$-group $D$ with $p \not= 2$, 
we identify $q$ with the bilinear form $b$.

\subsection{Local calculations}\label{subsection 7.1}
By the classification of finite quadratic forms \cite{Wa}, 
a finite quadratic form on an Abelian $p$-group of length $2$ is 
isometric to one of the following forms : 
\begin{eqnarray*}
&  A_{p, k}^{\theta , \theta '}   &  = 
          \begin{pmatrix}  \theta p^{-k} &   0  \\
                                           0                        &    \theta 'p^{-k} \end{pmatrix}  \: \: \: \: 
                  {\rm on\/} \: \:  ({\Z}/p^{k}{\Z})^{2},  \: \: \: \: 
                                     \theta , \theta ' \in {\Z}_{p}^{\times}/({\Z}_{p}^{\times})^{2} ,                                 \\
&  B_{p, l, k}^{\theta , \theta '}  & = 
          \begin{pmatrix}  \theta p^{-l} &   0  \\
                                           0                        &    \theta 'p^{-k} \end{pmatrix}  \: \: \: \: 
                  {\rm on\/} \: \:  {\Z}/p^{l}{\Z} \oplus {\Z}/p^{k}{\Z},  \: \: \: \: 
                                l>k, \: \: \:  \:    \theta , \theta ' \in {\Z}_{p}^{\times}/({\Z}_{p}^{\times})^{2} ,          \\
&  V_{k}  & = 
           \begin{pmatrix}  2^{1-k}     &   2^{-k}   \\
                                            2^{-k}        &   2^{1-k}   \end{pmatrix}  \: \: \: \: 
                  {\rm on\/} \: \:  ({\Z}/2^{k}{\Z})^{2},                                                                                                         \\
&  U_{k}   & =  
           \begin{pmatrix}  0              &   2^{-k}   \\
                                            2^{-k}    &  0              \end{pmatrix}  \: \: \: \: 
                  {\rm on\/} \: \:  ({\Z}/2^{k}{\Z})^{2}.                                                                      
\end{eqnarray*}
The following proposition is essentially a consequence of successive approximation.

\begin{prop}\label{Hensel lemma}
Let $(D, q)$ be a finite quadratic form on a $p$-group 
with no direct summand of order $2$. 
For the finite quadratic form 
\begin{equation*}\label{reduced quadratic form}
(\overline{D}, \bar{q}) = (D/N,  \: p\cdot q), \: \: \: \: \: N = \{ x \in D, \: px=0 \} ,  
\end{equation*}
the natural homomorphism 
\begin{equation*}\label{reducing isometries}
\kappa : O(D, q) \longrightarrow O(\overline{D}, \bar{q}) 
\end{equation*}
is surjective. 
\end{prop}

\begin{proof}
Note that 
$(\overline{D}, \bar{q})$ is well-defined as a finite quadratic form.  
A comparison of the classification of ${\Z}_{p}$-lattices and 
that of finite quadratic forms (cf. \cite{Ni} Proposition 1.8.1) 
enables us to find an even ${\Z}_{p}$-lattice $L$ such that 
$(D_{L(p)}, q_{L(p)}) \simeq (D, q)$.  
Then the homomorphism $\kappa $ is identified with 
the homomorphism 
\begin{equation*}
\kappa ' : O(D_{L(p)}, q_{L(p)}) 
= O(\frac{1}{p}L^{\vee}/L, \: p\cdot q_{L}) 
\to  O(\frac{1}{p}L^{\vee}/\frac{1}{p}L, \: p^{2}\cdot q_{L}) 
\simeq O(D_{L}, q_{L}). 
\end{equation*}
Consider the natural homomorphisms 
$O(L) \to O(D_{L})$  and 
$O(L(p)) \to O(D_{L(p)})$,  
which are surjective by Corollary 1.9.6 of \cite{Ni}. 
Since the diagram 
\begin{equation*}
\begin{array}{ccc} 
  O(L(p))                                             &                     =                                      &          O(L)                                         \\
   \downarrow                                   &                                                             &        \downarrow                            \\
  O(D_{L(p)})                                     &    \stackrel{\kappa '}{\to}                    &       O(D_{L}) ,  
\end{array}
\end{equation*}
commutes,   
we see that $\kappa '$ is surjective.  
\end{proof}

Thus natural reduction homomorphisms 
\begin{equation*}
O(A_{p, k}^{\theta , \theta '})  \to  O(A_{p, k-1}^{\theta , \theta '}), \: \: \: 
O(B_{p, l, k}^{\theta , \theta '})  \to  O(B_{p, l-1, k-1}^{\theta , \theta '}), \: \: \:  \cdots 
\end{equation*}
are defined and are surjective if 
$p \not= 2$ or $k \geq 2$.

\begin{lemma}\label{kernel of reduction general}
We have the following isomorphisms. 

$(1)$  ${\rm Ker\/} (O(A_{p, k}^{\theta , \theta '})  \twoheadrightarrow  O(A_{p, k-1}^{\theta , \theta '})) 
                 \simeq {\Z}/p{\Z}$ ,   
              where $k\geq 2$ if $p \not= 2$,  and  $k \geq 3$ if $p=2$.  

$(2)$  ${\rm Ker\/} (O(B_{p, l, k}^{\theta , \theta '})  \twoheadrightarrow  O(B_{p, l-1, k-1}^{\theta , \theta '}))  
                 \simeq {\Z}/p{\Z}$ ,  
              where $k\geq 2$ if $p \not= 2$,  and $k \geq 3$ if $p=2$.  

$(3)$  ${\rm Ker\/} (O(V_{k})  \twoheadrightarrow  O(V_{k-1})) \simeq {\Z}/2{\Z}$   
              where $k\geq 2$.  

$(4)$  ${\rm Ker\/} (O(U_{k})  \twoheadrightarrow  O(U_{k-1})) \simeq {\Z}/2{\Z}$  
              where $k\geq 2$.  
\end{lemma}

\begin{proof}
We prove only the assertion $(2)$.  
Other assertions can be proved analogously and are left to the reader. 
Let $p \not= 2$.  
An isometry 
$\gamma \in O(B_{p, l, k}^{\theta , \theta '})$ 
contained in the kernel of the reduction 
$O(B_{p, l, k}^{\theta , \theta '})  \to  O(B_{p, l-1, k-1}^{\theta , \theta '})$ 
is represented as 
\begin{equation*}
\gamma =         \begin{pmatrix}  1+ p^{l-1}a    &  p^{l-1}b           \\
                                                                 p^{k-1}c       &  1+ p^{k-1}d       \end{pmatrix} ,  \: \: \: 
						a, b, c, d \in {\Z}/p{\Z}.  
\end{equation*}
Since $\gamma $ preserves the quadratic form, 
we have 
\begin{equation}\label{isometry condition general}
{}^t\gamma 		 \begin{pmatrix}  \theta    &             0                     \\
                                                                       0        &   p^{l-k}\theta '       \end{pmatrix}   \gamma  
\equiv 																	   
                       \begin{pmatrix}  \theta    &             0                     \\
                                                              0        &   p^{l-k}\theta '       \end{pmatrix}   \: \: \: 
{\rm mod\/} \: 
                      \begin{pmatrix}  p^{l}{\Z}    &      p^{l}{\Z}              \\
                                                       p^{l}{\Z}       &   p^{l}{\Z}        \end{pmatrix}  	.														
\end{equation} 
Then trivial calculation shows that 
$a=d=0$ and $\theta b + \theta ' c=0$ 
so that 
the kernel is isomorphic to ${\Z}/p{\Z}$.  
For $p=2$,  
we need to replace $(\ref{isometry condition general})$ 
by the equation 
\begin{equation*}
{}^t\gamma 		 \begin{pmatrix}  \theta    &             0                     \\
                                                                       0        &   2^{l-k}\theta '       \end{pmatrix}   \gamma  
\equiv 																	   
                       \begin{pmatrix}  \theta    &             0                     \\
                                                              0        &   2^{l-k}\theta '       \end{pmatrix}   \: \: \: 
{\rm mod\/} \: 
                      \begin{pmatrix}  2^{l+1}{\Z}    &      2^{l}{\Z}              \\
                                                       2^{l}{\Z}          &   2^{l+1}{\Z}        \end{pmatrix}  	.														
\end{equation*}  
\end{proof}

\begin{lemma}\label{kernel of reduction interpolate}
We have the following isomorphisms. 
       
$(1)$  ${\rm Ker\/} (O(A_{2, 2}^{\theta , \theta '})  \twoheadrightarrow  O(A_{2, 1}^{\theta , \theta '})) 
                   \simeq    {\Z}/2{\Z} \oplus {\Z}/2{\Z}$.    
              
$(2)$  ${\rm Ker\/} (O(B_{2, l, 2}^{\theta , \theta '})  \twoheadrightarrow  O(B_{2, l-1, 1}^{\theta , \theta '})) 
                  \simeq    {\Z}/2{\Z} \oplus {\Z}/2{\Z}$.    
    				  
$(3)$ If $p \not= 2$, then 
            ${\rm Ker\/} (O(B_{p, l, 1}^{\theta , \theta '})  \twoheadrightarrow  O(B_{p, l-1, 0}^{\theta , \theta '}))$ 
         	is the dihedral group of order $2p$. 
\end{lemma}   
        
\begin{proof}  
We prove only the assertion $(3)$.  
An isometry    
$\gamma \in O(B_{p, l, 1}^{\theta , \theta '})$ 
contained in the kernel is represented as 
\begin{equation*}  
\gamma =         \begin{pmatrix}  1+ p^{l-1}a             &       p^{l-1}b           \\
                                                                           c            &             d       \end{pmatrix} ,  \: \: \: 
						a, b, c, d \in {\Z}/p{\Z}.  
\end{equation*}			 
By the isometry condition we have 
\begin{equation*}  
d^{2}=1, \; \; \;      
b= -\theta '' dc, \; \; \;     
a= -2^{-1}\theta '' c^{2},  
\end{equation*}   
where     
$\theta '' = \theta ^{-1} \theta '$.  
So there are ambiguities of  
$d \in \{  \pm  {\rm id\/}  \} $  
and    
$c \in {\Z}/p{\Z}$.   
It follows that the kernel is generated by     
\begin{equation*}     
\eta =      \begin{pmatrix}  1-2^{-1}\theta '' p^{l-1}         &      -\theta ''p^{l-1}               \\
                                                       1                                             &                      1                      \end{pmatrix}  
\: \: \: 		{\rm and\/}			   \: \: \: 														
\sigma   =      \begin{pmatrix}           1        &      0               \\
                                                                    0       &       -1                \end{pmatrix}  
\end{equation*}
with the relations 
$\eta ^{p} = \sigma ^{2} = (\eta \sigma)^{2} = 1$.  
\end{proof}

\begin{lemma}\label{initial values}
Let 
$\chi _{p}(a) = \left( \frac{a}{p} \right) $ 
be the Legendre symbol for $p\not= 2$. 																	

$(1)$ For $p \not= 2$,  
            $O(A_{p, 1}^{\theta , \theta '})$ 
         	is the dihedral group of order $2(p-\chi _{p}(-\theta \theta '))$. 
			 
$(2)$  We have 
\begin{equation*} 
                    O(A_{2, 1}^{\theta , \theta '})  \simeq 
                        \left\{ \begin{array}{cl} 
                                                 {\Z}/2{\Z},           &    \:  {\rm if\/} \: \:  \:  \theta \theta ' \equiv    1   \mod{4} , \\ 
                                                     \{ 1 \},               &     \:  {\rm if\/} \: \:  \:  \theta \theta ' \equiv   -1  \mod{4} . 
                                                \end{array} \right. 
             \end{equation*}

$(3)$  We have 
\begin{equation*} 
                    O(B_{2, l, 1}^{\theta , \theta '})  \simeq 
                        \left\{ \begin{array}{cl} 
                                                 {\Z}/2{\Z},                             &   \:   {\rm if\/} \: \:  \:  l=2,3 , \\ 
                                {\Z}/2{\Z} \oplus {\Z}/2{\Z} ,         &   \:   {\rm if\/} \: \:  \:  l\geq 4 . 
                                                \end{array} \right. 
             \end{equation*}

$(4)$  $O(V_{1})$ is the symmetric group $\frak{S} _{3}$.

$(5)$  $O(U_{1}) \simeq {\Z}/2{\Z}$.  
\end{lemma}

\begin{proof}
See Theorem 11.4 of \cite{Ta} for the assertion $(1)$.  
The verifications of the assertions $(2)$, $(4)$, $(5)$ are straightforward.  
We prove $(3)$.  
Let 
$\gamma =       \begin{pmatrix}         a       &      2^{l-1}b            \\
                                                                       c       &            d                    \end{pmatrix} 
\in  O(B_{2, l, 1}^{\theta , \theta '})$,  
where 
$a\in {\Z}/2^{l}{\Z}$ 
and 
$b, c, d \in {\Z}/2{\Z}$.  
If we denote 
$\theta '' := \theta ^{-1} \theta '$,  
the isometry condition for $\gamma $ is the following equations : 
\begin{eqnarray}
a^{2} + c^{2}\theta ''2^{l-1}            \equiv           1                   &    &     \mod{2^{l+1}},        \label{congruence 1}\\  
b^{2}2^{l-1} + d^{2}\theta ''             \equiv       \theta ''        &    &     \mod{4},                      \label{congruence 2}\\  
ab + cd \theta ''                                      \equiv         0                     &    &     \mod{2}.                      \label{congruence 3}
\end{eqnarray}
When $l=2$, 
we see that 
$a=\pm 1$, $b=c=0$, and $d=1$.  
When $l\geq 3$,  
we have 
$d=1$ by $(\ref{congruence 2})$  and 
$b=c$ by $(\ref{congruence 3})$. 
There are two possibilities for $c$ : $0$ or $1$.  
If $c=0$,  
then $a = \pm 1$ satisfy the equation $(\ref{congruence 1})$.  
If $c=1$, 
then the equation $(\ref{congruence 1})$ is written as 
\begin{equation}\label{congruence 1'}
a^{2} = 1-2^{l-1}\theta ''  \mod{2^{l+1}} . 
\end{equation}
When $l=3$,  
$(\ref{congruence 1'})$ does not have solution.  
When $l=4$,  
$(\ref{congruence 1'})$ has solutions $a= \pm (1+4\theta '')$. 
When $l\geq 5$, 
$(\ref{congruence 1'})$ has solutions $a= \pm (1-2^{l-2}\theta '')$. 
\end{proof}

From Lemmas \ref{kernel of reduction general}, \ref{kernel of reduction interpolate}, and \ref{initial values} 
we obtain the following results.

\begin{prop}\label{local result odd} 
Let $p\not= 2$ and $k\geq 1$.  

$(1)$   We have $| O(A_{p, k}^{\theta , \theta '})  |    = 2  \cdot p^{k-1}  \cdot (p-  \chi _{p}(-\theta \theta '))$.  

$(2)$   We have  $| O(B_{p, l, k}^{\theta , \theta '})  |  =  4 \cdot p^{k}$.  
\end{prop}

\begin{prop}\label{local result 2}
We have the following equalities.  

$(1)$ 
\begin{equation*}
| O(A_{2, k}^{\theta , \theta '})  | = 
                                      \left\{ \begin{array}{cc} 
                                                 2^{k},                   &     -\theta \theta ' \equiv  1   \mod{4} ,      \: k\geq 2    \\ 
                                                 2^{k+1},              &     -\theta \theta ' \equiv  -1   \mod{4} ,    \: k\geq 2    \\ 
												 1,                           &      -\theta \theta ' \equiv  1   \mod{4} ,      \: k=1           \\ 
												 2,                           &       -\theta \theta ' \equiv  -1   \mod{4} ,      \: k=1            
                                                \end{array} \right. 
\end{equation*}

$(2)$  
\begin{equation*}
| O(B_{2, l, k}^{\theta , \theta '})  | = 
                                      \left\{ \begin{array}{cc} 
                                                 2^{k+1},              &     l-k \leq 2,      \: k\geq 2    \\ 
                                                 2^{k+2},              &     l-k \geq 3,    \: k\geq 2    \\ 
												 2,                           &      l-k \leq 2,      \: k=1           \\ 
												 4,                           &      l-k \geq 3,      \: k=1            
                                                \end{array} \right. 
\end{equation*}

$(3)$ $| O(V_{k}) |  =  2^{k} \cdot 3$. 

$(4)$  $| O(U_{k}) |  =  2^{k}$. 
\end{prop}

\subsection{Global results}
Let $L$ be an even lattice of rank $2$.  
Denote by $(n, m)$ the invariant factor of $L \subset L^{\vee}$.  
That is, we have $n|m$ and 
there is a basis $\{ v_{1}, v_{2} \} $ of $L^{\vee}$ 
such that $L = \langle nv_{1}, mv_{2} \rangle $.  
Let 
\begin{eqnarray*}\label{prime decomp of inv factor}
n  &  =  &  p_{1}^{e_{1}} \cdots p_{\alpha}^{e_{\alpha}} \cdot   q_{1}^{f_{1}} \cdots q_{\beta}^{f_{\beta}},  \\ 
m &  =  &   p_{1}^{e_{1}} \cdots p_{\alpha}^{e_{\alpha}} \cdot   q_{1}^{f_{1}'} \cdots q_{\beta}^{f_{\beta}'} 
                      \cdot r_{1}^{g_{1}} \cdots r_{\gamma}^{g_{\gamma}},   
					  \: \: \: \:  f_{i}'>f_{i},  
\end{eqnarray*}
be the prime decompositions of $n$ and $m$.  
As groups, 
the $p$-components $D_{p}$ of the discriminant group $D_{L}$ are as follows : 
\begin{equation*}\label{p-comp of disc form and inv factor}
D_{p_{i}} \simeq ({\Z}/p_{i}^{e_{i}}{\Z})^{2},                                                         \: \: \: \: 
D_{q_{i}} \simeq  {\Z}/q_{i}^{f_{i}}{\Z} \oplus {\Z}/q_{i}^{f_{i}'}{\Z},          \: \: \: \: 
D_{r_{i}} \simeq  {\Z}/r_{i}^{g_{i}}{\Z}.  
\end{equation*} 
For a prime number $p$ dividing $n$, put 
\begin{equation*}\label{certain units}
\varepsilon _{p} :=  
                                      \left\{ \begin{array}{cc} 
                                                 -p^{-2e_{i}}\cdot {\rm det\/}(L),                  &      p=p_{i}. \\ 
                                                 0,                                                                               &      p=q_{i}. 
                                                \end{array} \right. 
\end{equation*}
When 
$D_{p} \simeq A_{p, e}^{\theta , \theta '}$ as a quadratic form, 
we have 
$-\theta \theta ' \equiv \varepsilon _{p}  \in  {\Z}_{p}^{\times}/({\Z}_{p}^{\times})^{2}$.  
When $D_{2} \simeq U_{e}$ (resp. $V_{e}$),  
we have 
$\varepsilon _{2} \equiv 1$ (resp. $5$) $\mod{8}$.

For an odd prime number $p$ and $\varepsilon \in \mathbb{F}_{p}^{\times}$, 
let 
$\chi _{p} (\varepsilon)= \left(   \frac{\varepsilon}{p}\right) $ be the Legendre symbol.  
We put 
$\chi _{p} (0) := 0$.   
For a natural number $n \equiv 1 \mod{4}$,  
we define 
\begin{equation*}
\chi _{2}(n) =  \left\{ \begin{array}{cl} 
				      1,                 &   \: \: \:   n \equiv 1 \mod{8},			\\ 
                                           -1,                &   \:  \: \:  n \equiv 5 \mod{8}.      			
                                \end{array} \right. 
\end{equation*}   
For a natural number $N$ 
let $\widetilde{\tau}(N) := \tau (N)$ if $N>1$ 
and $\widetilde{\tau}(1) := 0$. 
Then we have 
\begin{equation*}
\widetilde{\tau} (n) + \widetilde{\tau} (n^{-1}m) = \alpha + 2\beta + \gamma .
\end{equation*}


We are now in a position to express the formula for $|O(D_{L})|$.  

\begin{theorem}\label{order of isom grp of disc form}
Let $L, (n, m), \chi _{p}$, and $\widetilde{\tau}$ be as above. 

$(1)$ 
If $D_{2}$ is either trivial or $U_{k}$ or $V_{k}$,  
or equivalently if $L\simeq M(2^{e})$ for an even lattice $M$ with ${\rm det\/}(M)$ odd, 
then 
\begin{equation*}
|O(D_{L})| = 
2^{\widetilde{\tau} (n) + \widetilde{\tau}(n^{-1}m)} \cdot n \cdot  \mathop{\prod}_{p|n} \Bigl(  1- \frac{\chi _{p}(\varepsilon _{p})}{p} \Bigr) .  
\end{equation*}

$(2)$ 
If $D_{2}\simeq A_{2, k}^{\theta , \theta '}$,  
or equivalently if $L\simeq M(2^{e})$ for an odd lattice $M$ with ${\rm det\/}(M)$ odd, 
then  
\begin{equation*}
|O(D_{L})| = 
C \cdot 2^{\widetilde{\tau} (2^{-1}n) + \widetilde{\tau} (n^{-1}m)} \cdot n \cdot  
\mathop{\prod}_{{p|n}\atop{p\not= 2}} \Bigl(  1- \frac{\chi _{p}(\varepsilon _{p})}{p} \Bigr) ,   
\end{equation*}
where $C=1$ if $\varepsilon _{2} \equiv -1 \mod{4}$, 
$C=\frac{1}{2}$ if $\varepsilon _{2} \equiv 1 \mod{4}$.

$(3)$ 
If $D_{2}\simeq B_{2, l, k}^{\theta , \theta '}$,  
or equivalently if $L\simeq M(2^{e})$ for an odd lattice $M$ with ${\rm det\/}(M)$ even,  
then  
\begin{equation*}
|O(D_{L})| = 
C \cdot 2^{\widetilde{\tau} (2^{-1}n) +\widetilde{\tau} (n^{-1}m)} \cdot n \cdot  
\mathop{\prod}_{{p|n}\atop{p\not= 2}} \Bigl(  1- \frac{\chi _{p}(\varepsilon _{p})}{p} \Bigr) ,   
\end{equation*}
where $C=1$ if $l-k\geq 3$, 
$C=\frac{1}{2}$ if $l-k \leq 2$.
\end{theorem}

\begin{proof}  
This follows immediately from the results of the previous section $\ref{subsection 7.1}$.      
Note that $D_{2}$ is never cyclic.                  
\end{proof}

From section $\ref{subsection 7.1}$ 							  
we also deduce the following.

\begin{prop}\label{multiple-divisible}
Let $L$ be a rank $2$ even lattice and $n>1$ be a natural number.  

$(1)$   
The number $| O(D_{L(n)}) |$ is divisible by the number $| O(D_{L}) |$.

$(2)$ 
If $n$ is coprime to ${\rm det\/}(L)$, 
we have 
\begin{equation*}\label{eqn: multiple-divisibility coprime case}
| O(D_{L(n)}) |  =  
| O(D_{L}) | \cdot 
2^{\tau (n)} \cdot n \cdot \mathop{\prod}_{p|n} \left(  1 - \frac{\chi _{p}(-{\rm det\/}(L))}{p}    \right) .  
\end{equation*} 

$(3)$  
If $L$ is primitive and $n|{\rm det\/}(L)^{a}$ for some $a \in {\Z}_{>0}$, 
we have 
\begin{equation*}\label{eqn: multiple-divisibility totally non-coprime case}
| O(D_{L(n)}) |  =  
| O(D_{L}) | \cdot 2^{\tau (n)} \cdot n.  
\end{equation*} 
\end{prop}

\begin{proof}
The assertion $(1)$ follows from Proposition \ref{Hensel lemma}.  

$(2)$ 
We identify the ${\Z}$-modules underlying $L$ and $L(n)$ in a natural way. 
Since $n$ is coprime to $| D_{L} |$, 
we have the orthogonal decomposition 
\begin{equation*}
D_{L(n)} = (L^{\vee}/L) \oplus (n^{-1}L/L). 
\end{equation*}
Hence we have 
\begin{equation*}
O(D_{L(n)}) = O(D_{L}) \oplus \mathop{\bigoplus}_{p|n} O(p^{-e}L/L),  
\end{equation*}
where $n=\prod p^{e}$ is the prime decomposition of $n$.  
If $p$ is odd, then 
$p^{-e}L/L \simeq A_{p, e}^{1, {\rm det\/}(L)}$.  
If $p=2$, 
we have 
$2^{-e}L/L \simeq U_{e} \; \text{or} \; V_{e}$ 
according to 
$-{\rm det\/}(L) \equiv 1 \; \text{or} \; 5 \mod{8}$.  

$(3)$ 
Let $D_{L} = \oplus _{p} D_{p}$ be the decomposition into $p$-components. 
For $p \not= 2$, $D_{p}$ is cyclic.  
On the other hand, $D_{2}$ is either trivial or $A_{2, 1}^{\theta , \theta '}$ or $B_{2, l, 1}^{\theta , \theta '}$.  
Thus our claim follows from Lemmas \ref{kernel of reduction general} and \ref{kernel of reduction interpolate}.  
\end{proof}

\noindent\textbf{Acknowledgements.}   
The author wishes to thank 
Professors Herbert Lange, Matthias Schuett,  Ichiro Shimada, Tetsuji Shioda       
for valuable comments.   
The contents of Sections 5.2 and 6.2 were improved considerably due to Professor Shioda's  advice. 
The author is specially grateful to Professor Ken-Ichi Yoshikawa for encouragements and advices.  
This work was supported by Grant-in-Aid for JSPS fellows.

\end{document}